\documentclass[a4paper,16pt]{article}
\usepackage{amsmath}
\usepackage{amssymb}
\usepackage{amsthm}
\usepackage{geometry}
\usepackage{mathrsfs}
\usepackage{cite}
\usepackage[pdftex]{hyperref}
\usepackage{enumerate}
\def\N{\mathbb{N}}
\def\R{\mathbb{R}}
\def\Z{\mathbb{Z}}
\def\bP{\mathbb{P}}
\def\E{\mathbb{E}}

\def\T{\mathbb{T}}
\def\vep{\varepsilon}
\usepackage[pdftex]{hyperref}
\hypersetup{CJKbookmarks=true}

\theoremstyle{plain}
\newtheorem{theorem}{Theorem}[section]

 \newtheorem{proposition}[theorem]{Proposition}

 \newtheorem{lemma}[theorem]{Lemma}

\newtheorem{remark}{Remark}[section]

  \makeatletter
  \@addtoreset{equation}{section}
  \makeatother
\author{Lianghui Luo\footnote{Institut de Math\'ematiques de Toulouse, UMR5219. Universit\'e de Toulouse; CNRS. UPS, F-31062 Toulouse Cedex 9, France. E-mail: lianghui.luo@math.univ-toulouse.fr}}
\title{Precise upper deviation estimates for the maximum of a branching random walk}
\begin{document}
\maketitle
\begin{abstract}
We consider the precise upper large deviations estimates for the maximal displacement of a branching random walk. In addition, we obtain a description of the extremal process of the branching random walk conditioned on this large deviations event. This introduces a family of point measure playing a role similar to the decoration measures introduced in \cite{bovier} for branching Brownian motion.
\end{abstract}

\section{Introduction}
We consider a branching random walk on the real line, which can be described as follows. At time 0, one particle is located at 0. At each time $n+1$, each particle alive at time $n$ dies and produces its children. The displacement of the children with respect to the position of their parent is distributed as a point process $\mathcal{L}$ on $\R$, and is independent of other branching events in the process.

For each particle $u$, we write $V(u)$ for its position and $|u|$ for its generation. Let $M_n$ be the maximal position occupied by a particle at time $n$. The asymptotic property of $M_n$ as $n$ go to infinity has been well studied. Hammersley\cite{hammersley}, Kingman\cite{kingman} and Biggins\cite{biggins} obtained, under increasingly general conditions,the first order speed of $M_n$ and the existence of an explicit constant $x^*\in\R$, depending on the reproduction law, satisfying
\[\frac{M_n}{n}\to x^*,\quad\text{as }n\to\infty,\quad \text{almost surely on the non-extinction set}.\]
The asymptotic fluctuations of $M_n-x^*n$ were studied by Hu and Shi\cite{hu}, Addario-Berry and Reed\cite{berry}, and Bramson and Zeitouni\cite{bramson}. Under mild conditions, they showed that there exists $\theta^*>0$ such that $M_n-x^*n$ fluctuates almost surely between $-\frac{3}{2\theta^*}\log n$ and $-\frac{1}{2\theta^*}\log n$, and $M_n-\mathrm{Med}(M_n)$ is tight, where $\mathrm{Med}(M_n)=x^*n-\frac{3}{2\theta^*}\log n+O(1)$. Finally, A\"{i}d\'ekon\cite{aidekon2} proved that $M_n-x^*n+\frac{3}{2\theta^*}\log n$ converges in law to a randomly shifted Gumbel variable. For more background about branching random walk, we refer to the lecture notes of Shi\cite{shi}.

In this article, we take interest in the upper deviation probability of $M_n$. Under the assumption that the displacements of the children of an individual are independent of one another and of their number, Gantert and H\"ofelsauer\cite{gantert} proved that for $x>x^*$, the upper large deviation probability $\bP(M_n\geq nx)$ decays exponentially if the displacement satisfies an exponential moment condition. We also mention that when the number of children is non-random, precise estimates on the right tail of $M_n$ have been obtained in \cite{buraczewski}. Without the independence between displacement and offspring, under some strong conditions, Rouault\cite{rouault} showed the precise estimation of the probability $\bP(\exists |u|=n,\ \text{s.t.}\ V(u)\in [nx-\delta,nx+\delta])$ for $x>x^*$ and $\delta>0$. For the lower deviation probability estimates of maximum, see \cite{hu1}, \cite{gantert}, \cite{chen} and \cite{ghosh}. For a branching random walk in which displacements have semi-exponential tails, Dyszewski et al.\cite{dyszewski} proved that the law of $M_n/n^\alpha$ satisfies a large deviation principle, for some $\alpha\in(0,1)$ depending on the tail distribution of the displacement. When the displacement has regularly varying tails, Bhattacharya\cite{bhattacharya} showed that there exists a sequence $(\gamma_n)_{n\geq 1}$ such that $M_n/\gamma_n$ satisfies a large deviation principle.

In the context of branching Brownian motion, which is a continuous-time analogue of the branching random walk, in which particles move according to i.i.d Brownian motion while giving birth to children at an exponential rate, the precise estimation of the upper deviation probability of the maximum were studied in \cite{chauvin}.

 Inspired by \cite{gantert},\cite{buraczewski} and \cite{berestycki}, in this article, we want to obtain an analogous result for the branching random walk, i.e. computing the asymptotic behavior of $\bP(M_n\geq nx)$ for $x>x^*$, under close to optimal conditions. Before showing the main results, we introduce some necessary notation.
Consider a branching random walk started from one particle at initial time. With previous notation, $(V(u):|u|=n)$ are the positions of particles alive at generation $n$.
 Define
 \begin{equation}\label{Zn}
 Z_n:=\sum_{|u|=n}\delta_{V(u)}
 \end{equation}
 the counting measure formed by the positions of the particles of the branching random walk at time $n$. For $A\subset\R$, we write $Z_n(A)$ for the number of particles located in the set $A$ at time $n$. We recall that $M_n=\max_{|u|=n}V(u)$ is the maximal displacement at time $n$, with the usual convention $M_n=-\infty$ if $\{|u|=n\}=\varnothing.$

We denote by
\[\psi(\theta):=\log \E\left(\sum_{|u|=1}e^{\theta V(u)}\right)\]
the log-Laplace transformation of the reproduction law of the branching random walk. We refer to the branching random walk as supercritical, critical or subcritical respectively when $\psi(0)>0$, $\psi(0)=0$ or $\psi(0)<0$, i.e. the mean number of the children is larger than, equal to or smaller than 1.
To avoid that the initial particle produces 0 particle almost surely, we always assume that
\[\bP(Z_1(\R)=0)<1\]
in this article.

Additionally, we assume that there exists some $\theta>0$ satisfying $\psi(\theta)<\infty$, and we define
\[\psi^*(x):=\sup_{\theta\geq 0}\{\theta x-\psi(\theta)\}\in\R\cup\{+\infty\}\]
 the Legendre transformation of $\psi$, which is non-decreasing and convex on $\R$. It is known from \cite{biggins} that in the supercritical case, almost surely on the non-extinction set,
\[\frac{M_n}{n}\to x^*:=\inf_{\theta>0}\frac{\psi(\theta)}{\theta}=\sup\{x\in\R:\psi^*(x)<0\},\quad\text{as }n\to\infty.\]

Let us state the upper deviation estimates of $M_n$, which is a slight refinement of \cite[Theorem 3.2]{gantert}.
\begin{proposition}\label{log}
Assume that there exists some $\theta>0$ satisfying $\psi(\theta)<\infty$. If $\psi(0)>0$, we have
\[\lim_{n\to\infty}\frac{1}{n}\log\bP(M_n\geq nx)=-\psi^*(x),\quad\text{for any }x>x^*.\]
If $\psi(0)\leq 0$, we have
\[\lim_{n\to\infty}\frac{1}{n}\log\bP(M_n\geq nx)=-\psi^*(x),\quad\text{for any }x>\frac{\E(\sum_{|u|=1}V(u))}{\E(Z_1(\R))}\in\R\cup\{-\infty\}.\]
\end{proposition}
\begin{remark}
This proposition shows that $M_n$ satisfies a large deviation principle as soon as the intensity measure of the reproduction law of the branching random walk admits an exponential moment. In particular, no extra condition on the reproduction law is needed.
\end{remark}
In order to obtain the precise estimation of the probability $\bP(M_n\geq nx)$, we need some extra assumptions.
For the offspring law of the branching random walk, we assume that
\begin{equation}\label{asn}
\bP(Z_1(\R)\geq 2)>0,
\end{equation}
which implies that there is indeed some branching in the branching process we consider. We remark that if (\ref{asn}) does not hold, large deviation estimates of $M_n$ build down to large deviations for an usual random walk, obtained by Bahadur and Rao \cite{bahadur} and Petrov\cite{Petrov}, see forthcoming Remark \ref{randomwalk}.
We assume that there exists $\theta>0$ such that
\begin{equation}\label{as1}
\psi(\theta)<\infty \quad\text{and}\quad \theta\psi'(\theta)>\psi(\theta),
\end{equation}
where we denote by
\[\psi'(\theta):=\E\left(\sum_{|u|=1}V(u)e^{\theta V(u)-\psi(\theta)}\right), \]
which is assumed to be well-defined. Observe that if $\psi$ is finite in a neighbourhood of $\theta$, then by Lebesgue's dominated convergence theorem, $\psi'(\theta)$ corresponds to the derivative of $\psi$ at $\theta$, justifying the notation.
And we assume that
\begin{equation}\label{as2}
\quad\sigma^2:=\E\left(\sum_{|u|=1}(V(u)-\psi'(\theta))^2e^{\theta V(u)}\right)\in(0,\infty).
\end{equation}
Moreover, we add a mild $L\log L$-type integrable condition
\begin{equation}\label{as3}
\E\left(\sum_{|u|=1}e^{\theta V(u)}\log_+\sum_{|u|=1}e^{\theta V(u)}\right)<\infty,
\end{equation}
where $\log_+(x):=\log(\max\{x,1\}),$ for $x\geq 0$.
Finally, we assume that the branching random walk is non-lattice, i.e. for any $a,b\in\R$,
\begin{equation}\label{as4}
\bP(V(u)\in a+b\Z, \forall |u|=1 )<1,
\end{equation}
where $a+b\Z:=\{a+bn:n\in\Z\}$.

The next theorem gives the precise asymptotic behavior of $\bP(M_n\geq n\psi'(\theta)+y)$ as $n\to\infty$, uniformly in $y=O(\sqrt{n})$. It extends the results showed in \cite{chauvin} for branching Brownian motion to the branching random walk.
\begin{theorem}\label{th2}
Assume (\ref{asn}), (\ref{as1}), (\ref{as2}), (\ref{as3}) and (\ref{as4}). There exists an explicit constant $C(\theta)\in(0,1)$ defined in (\ref{constantC}) such that for any positive sequence $(a_n)_{n\geq0}$ satisfying $a_n=O(\sqrt{n})$, we have
\begin{equation}\label{uniform}
\lim_{n\to\infty}\sup_{|y|\leq a_n}\left|\sqrt{n}e^{\frac{y^2}{2\sigma^2 n}}e^{\theta y}e^{n(\theta \psi'(\theta)-\psi(\theta))}\bP(M_n\geq n\psi'(\theta)+y)- \frac{C(\theta)}{\sqrt{2\pi}\sigma\theta}\right|=0.
\end{equation}
 In particular, we have
\[\lim_{n\to\infty}\frac{\bP(M_n\geq n\psi'(\theta))}{\E(Z_n([n\psi'(\theta),\infty)))}= C(\theta).\]
\end{theorem}
\begin{remark}\label{randomwalk}
If assumption (\ref{asn}) does not hold, the conclusions in Theorem \ref{th2} remain valid with $C(\theta)=1$. Indeed if $\bP(Z_1(\R)\leq 1)=1$, we have
\[\bP(M_n\geq n\psi'(\theta)+y)=(\E(Z_1(\R)))^n\bP(S_n\geq n\psi'(\theta)+y),\]
where $(S_n)_{n\geq0 }$ is a random walk started from $0$ and whose step distribution is such that for any non-negative measurable function $h$, $\E(h(S_1))=(\E(Z_1(\R)))^{-1}\E(\sum_{|u|=1}h(V(u)))$. Therefore by the large deviations principle of \cite[Theorem 2]{Petrov}, we obtain (\ref{uniform}) with $C(\theta)=1$. In other words, we have
\[C(\theta)<1\quad \text{if and only if}\quad \bP(Z_1(\R)\geq 2)>0.\]
\end{remark}
With similar strategies as in \cite{berestycki} and \cite{rouault}, we use the spine decomposition theorem, time reversal and the local limit theorem to prove Theorem \ref{th2}. As a by-product, we obtain the joint limit distribution of the extremal process $\mathcal{E}_n:=\sum_{|u|=n}\delta_{V(u)-M_n}$ and $M_n-n\psi'(\theta)$ conditionally on the event $\{M_n\geq n\psi'(\theta)\}$.

\begin{theorem}\label{th1}
Assume (\ref{as1}), (\ref{as2}), (\ref{as3}) and (\ref{as4}). There exists a point process $D(\theta)$ such that $(\mathcal{E}_n,M_n-n\psi'(\theta))$ conditionally on the event $\{M_n\geq n\psi'(\theta)\}$ converges in law to $(D(\theta),\mathbf{e})$, where $\mathbf{e}$ is an exponential random variable of parameter $\theta$ independent of $D(\theta)$.
\end{theorem}
\begin{remark}
In particular, if assumption (\ref{asn}) does not hold, then $\mathcal{E}_{n}=\delta_0$ on $\{M_n>-\infty\}$, which implies that $D(\theta)=\delta_0$.
\end{remark}
In order to prove Theorem \ref{th1}, it is sufficient to demonstrate that for any non-negative continuous function $\phi$ with compact support and $x\geq 0$,
 \[\E(e^{-\langle\mathcal{E}_n,\phi\rangle}1_{\{M_n-n\psi'(\theta)\geq x\}}|M_n\geq n\psi'(\theta))\to e^{-\theta x}\E(e^{-\langle D(\theta),\phi\rangle})\]
 which is shown in Section \ref{th23}. The point process $D(\theta)$ obtained in Theorem \ref{th1} plays an analogous role to the decorations obtained in \cite[Propositon 7.5]{bovier} in the context of branching Brownian motion.

We also take interest in some properties of $D(\theta)$. Note that we have not assumed that the branching random walk is supercritical in Theorem $\ref{th1}$. If the branching random walk is subcritical (critical), we prove that the measure $D(\theta)$ is finite (infinite) respectively. More precisely, $\bP(D(\theta)(\R)<\infty)=1$ if and only if $\psi(0)<1$.

Besides, under some additional conditions, we show that $D(\theta)$ is continuous in distribution with respect to $\theta$, i.e. for $\theta_0>0$ such that $\psi'(\theta_0)>x^*$ and any non-negative continuous function $\phi$ with compact support, we have
\[\lim_{\theta\to\theta_0}\E(e^{-\langle D(\theta),\phi\rangle})=\E(e^{-\langle D(\theta_0),\phi\rangle}).\]

The rest of the paper is organized as follows. In Section \ref{pre}, we introduce Ulam-Harris-Neveu notation for trees and the many-to-one formula. Moreover, we introduce the spine decomposition, construct an auxiliary point process $D_n^\theta$ by time-reversal along the spine and show that the limit of $D_n^\theta$ exists. In Section \ref{app}, we study the asymptotic properties of $\E(e^{-\langle \mathcal{E}_n,\phi\rangle}1_{\{M_n\geq n\psi'(\theta)+y\}})$, and we see that $D(\theta)$ is actually the limit of $D_n^\theta$ biased by a random variable. In Section \ref{p1th23}, we prove Theorem \ref{th2}, Theorem \ref{th1} and Proposition \ref{log}. In Section \ref{finitenumber}, we discuss the number of atoms of $D(\theta)$ as the branching random walk is subcritical or critical, and we prove the continuity of $D(\theta)$ with respect to $\theta$ under additional integrability conditions.

In the proof of this paper, we denote by $C$ positive constant that may change from line to line.

\section{Spine decomposition from the tip of the branching random walk}\label{pre}
In this section, we introduce the many-to-one formula and spine decomposition theorem, and we construct an auxiliary and show the relationship between the extremal process and the auxiliary process. At the end, we give a random walk estimate by local limit theorem.
\subsection{Ulam-Harris-Neveu notation}\label{uhn}
We introduce the Ulam-Harris-Neveu notation for plane trees. Define the set of finite sequence of positive integer
\[\mathbb{T}:=\bigcup_{n\geq 0}\N^{n},\]
with the usual convention $\N^0:=\{\varnothing\}$. For $u\in \N^{n}$, we write $u=(u^{(1)},\ldots,u^{(n)})$ and
\begin{itemize}
\item we denote by $|u|$ the generation of $u\in\mathbb{T}$, i.e. if $u\in \N^{n}$, then $|u|:=n$.
\item for $k\leq |u|$, we use $u_k=(u^{(1)},\ldots,u^{(k)})$ to express the ancestor at generation $k$ of $u$. In particular, $u_0:=\varnothing$.
\item we denote by $\ddag\varnothing,u\ddag:=\{u_0,\ldots,u_{|u|}\}$ the set of vertices(including $\varnothing$ and $u$) in the unique shortest path connecting the root $\varnothing$ to $u$.
\end{itemize}

 Let $\{Z^{(u)}:=(z_k^{(u)}:k\geq 1):u\in\T\}$ be i.i.d random non-increasing sequences with $z_k^{(u)}\in\R\cup\{-\infty\}$ and be such that $\sum_{i\geq 1}\delta_{z_i^{(u)}}$ is distributed as the point process $\mathcal{L}$, with the convention that $\delta_{-\infty}:=0$.
 A branching random walk $Z:=\{V(u):u\in \T\}$ can be defined as follows: we set $V(\varnothing):=0$ and
 \[V(u):=\sum_{i=1}^{|u|}z_{u^{(i)}}^{(u_{i-1})},\quad\text{if }u\neq \varnothing,\]
with the convention that $-\infty+x=-\infty$ for any $x\in\R$. We treat particles $u$ such that $V(u)=-\infty$ as dead particles, which do not contribute to the branching random walk.

  We end this section with many-to-one formula, which is helpful for our proof.
 \begin{lemma}[Theorem 1.1 in\cite{shi}]
 Assume that $\theta>0$ is such that $\psi(\theta)<\infty$. For any $n\ge 1$ and any measurable function $g:\R^n\to[0,\infty)$, we have
 \[\E\left(\sum_{|u|=n}g(V(u_1),\ldots,V(u_n)\right)=\E(e^{-\theta S_n+n\psi(\theta)}g(S_1,\ldots,S_n )),\]
 where $(S_n)_{n\geq 0}$ is a random walk started from 0 with step distribution satisfying that for any nonnegative measurable function $h$,
 \[\E(h(S_1))=\E\left(\sum_{|u|=1}h(V(u))e^{\theta V(u)-\psi(\theta)}\right).\]
 \end{lemma}
 In the above lemma, the sum over $|u|=n$ is taken over all live particles in the branching random walk alive at time $n$.

\subsection{Spine decomposition: a change of measure}\label{sd}
Lyons\cite{lyons} introduced the spine decomposition theorem for the branching random walk as an alternative description of the law of the branching random biased by its additive martingale. This method is an extension of the spine decomposition obtained by Lyons et al.\cite{lyons1} for Galton-Watson processes and Chauvin and Rouault\cite{chauvin} for branching Brownian motion.

Let $(\mathscr{F}_n)_{n\geq 0}$ be the natural filtration of the branching random walk $Z$, i.e.
 \[\mathscr{F}_n:=\sigma (V(u):|u|\leq n),\]
 with $\mathscr{F}_\infty:=\sigma(V(u):u\in\mathbb{T})$.
As $\psi(\theta)<\infty$, consider an additive martingale $W_n^\theta$, defined by
\[W_n^\theta:=\sum_{|u|=n}e^{\theta V(u)-n\psi(\theta)}.\]
According to Kolmogorov's extension theorem, there exists a probability $\bP_Q$ on $\mathscr{F}_\infty$ satisfying for any $n\geq 0$, $A\in \mathscr{F}_n$,
\[\bP_Q(A)=\E(W_n^\theta1_{A}).\]

The spine decomposition consists in an alternative description of the law $\bP_Q$ constructed as a branching random walk with a spine.
Let a random pair $(\hat{\mathcal{L}},\xi)$ be such that for any non-negative measurable function $f$,
\[\E(f(\hat{\mathcal{L}},\xi))=\E\left(\sum_{k\geq 1}e^{\theta V(k)-\psi(\theta)}f(Z_1,k)\right).\]
 A branching random walk $(V(u):u\in\mathbb{T})$ with a spine $(\xi_n:n\geq 0)$ can be described as follows. At time 0, one particle $\xi_0:=\varnothing$ locates at position $V(\xi_0)=0$. At each time $n\geq1$, all particles die, while giving birth independently to sets of new particles. The displacements (with respect to their parent) of children of normal particle $z$ are distributed as $(V(i):i\geq 1)$. The displacements (with respect to their parent) of children of spine particle $\xi_{n-1}$ are distributed as $\hat{\mathcal{L}}$; the particle $\xi_n$ is chosen among the children $y$ of $\xi_{n-1}$ with probability proportional to $e^{\theta V(y)}$. Let us denote by $\tilde{\bP}$ be the law of the branching random walk with a spine. The spine decomposition theorem corresponds in the identification of the law of the size-biased branching random walk and the law of the branching random walk with a spine.
 \begin{theorem}[Lyons\cite{lyons}]
 For any $n\geq 0$, the law of $(V(u):|u|\leq n)$ is identical under $\bP_Q$ and $\tilde{\bP}$. Moreover, for any $|u|=n$,
\begin{equation}\label{spine}
\tilde{\bP}(u=\xi_n|\mathscr{F}_n)=\frac{e^{\theta V(u)-n\psi(\theta)}}{W_n^\theta},\quad \text{a.s.}.
\end{equation}
 \end{theorem}

Next, we introduce some notation that is helpful for constructing an auxiliary point process $D_n^\theta$. Suppose that $\{((b_k(i):i\geq 1),w^{(k)})\}_{k\geq1}$ are i.i.d. copies of $((\hat{\ell}_i:i\geq 1),\xi)$, where $(\hat{\ell}_i)_{i\geq 1}$ is the ranked sequence of atoms in $\hat{\mathcal{L}}$, which is non-increasing sequence converging to $-\infty$. Let $\{(V^{(i,k)}(u):u\in\T):k\geq 1,i\geq1\}$ be i.i.d copies of the branching random walk $(V(u):u\in \T)$, independent of $\{((b_k(i):i\geq 1),w^{(k)})\}_{k\geq1}$. For $n\geq 0$, we define an auxiliary point process
\begin{equation}\label{d1}
D^\theta_n:=\delta_0+\sum_{k=1}^{n}\sum_{i\neq w^{(k)}}\sum_{|u|=k-1}\delta_{b_k(i)+V^{(i,k)}(u)-S_{k}},
\end{equation}
with the convention $D^\theta_0:=\delta_0$, where
\[S_n:=\sum_{k=1}^n b_k(w^{(k)}).\]
 Recall that $\mathcal{E}_n=\sum_{|u|=n}\delta_{V(u)-M_n}$. By decomposing the branching random walk with spine along the path $\ddag \varnothing,\xi_n\ddag$ and reversing the spine and the displacement (with respect to their parent) of the children of the spine particles before $n$th generation, we show the relationship between $(\mathcal{E}_n,M_n-n\psi'(\theta))$ and $(D_n^\theta,S_n-n\psi'(\theta))$.
\begin{lemma}\label{lemD}
Let $\theta>0$ such that $\psi(\theta)<\infty$ and $\psi'(\theta)<\infty$. For any non-negative measurable functions $F$, $f$ and $n\geq 1$, we have
\begin{equation}\label{eqD}
\begin{split}
&\E(F(\mathcal{E}_n)f(M_n-n\psi'(\theta))1_{\{Z_n(\R)>0\}})\\
=&e^{n(\psi(\theta)-\theta\psi'(\theta))}\E\left(e^{-\theta (S_n-n\psi'(\theta))}F(D_n^\theta)f(S_n-n\psi'(\theta))\frac{1}{D_n^\theta(\{0\})}1_{\{D_n^\theta((0,\infty))=0\}}\right).
\end{split}
\end{equation}
\end{lemma}
\begin{proof}
We have
\[\E\left(\sum_{|u|=n}e^{\theta V(u)}1_{\{M_n=V(u)\}}\right)\leq \E\left(\sum_{|u|=n}e^{\theta V(u)}\right)= e^{n\psi(\theta)}<\infty,\]
which implies that $\bP(\mathcal{E}_n(\{0\})<\infty)=1.$
By definition, on $\{Z_n(\R)>0\}$, we have $\mathcal{E}_n(\{0\})\geq 1$.
For all non-negative measurable functions $F$ and $f$, we have
\[
\begin{split}
&\E[F(\mathcal{E}_n)f(M_n-\psi'(\theta)n)1_{\{Z_n(\R)>0\}}]\\
=&\E\left[\frac{1}{\mathcal{E}_n(\{0\})}\sum_{u\in N_n}F(\mathcal{E}_n^u)f(M_n-\psi'(\theta)n)1_{\{Z_n(\R)>0\}}\right]\\
=&\E_Q\left[\frac{1}{W_n^\theta}\frac{1}{\mathcal{E}_n(\{0\})}\sum_{u\in N_n}F(\mathcal{E}_n^u)f(V(u)-\psi'(\theta)n)\right],
\end{split}
\]
where $N_n:=\{|u|=n: V(u)=M_n\}$, and $\mathcal{E}_n^u:=\sum_{|v|=n}\delta_{V(v)-V(u)}$ is the branching random walk seen form the particle $u$.
Using the spine decomposition theorem, we have
\[
\begin{split}
&\E_Q\left[\frac{1}{W_n^\theta}\frac{1}{\mathcal{E}_n(\{0\})}\sum_{u\in N_n}F(\mathcal{E}_n^u)f(V(u)-\psi'(\theta)n)\right]\\
=&\tilde{\E}\left[\frac{1}{\mathcal{E}_n(\{0\})}\sum_{u\in N_n}F(\mathcal{E}_n^u)e^{-\theta V(u)+n\psi(\theta)}1_{\{\xi_n=u\}}f(V(u)-\psi'(\theta)n)\right]\\
=&\tilde{\E}\left[\frac{1}{\mathcal{E}_n^{\xi_n}(\{0\})}e^{-\theta V(\xi_n) +n\psi(\theta)}F(\mathcal{E}_n^{\xi_n})f(V(\xi_n)-\psi'(\theta)n)1_{\{\xi_n\in N_n\}}\right]\\
=&\tilde{\E}\left[\frac{1}{\mathcal{E}_n^{\xi_n}(\{0\})}e^{-\theta V(\xi_n) +n\psi(\theta)}F(\mathcal{E}_n^{\xi_n})f(V(\xi_n)-\psi'(\theta)n)1_{\{\mathcal{E}_n^{\xi_n}((0,\infty))=0\}}\right].
\end{split}
\]
Next, we construct a mapping from $\N^n$ to the atoms of $D_n^\theta$. Let $w^{(n)}_k:=(w^{(n)},\ldots,w^{(n-k+1)}),1\leq k\leq n$.
For $u\in \N^n$, define $r(u):=\sup\{k\geq 0:u_k=w^{(n)}_{k}\}$ and
\[V^{(n)}(u):=S_n-S_{n-r(u)}+b_{n-r(u)}(u^{(r(u)+1)})+V^{(u^{(r(u)+1)},n-r(u))}(u|_{n-r(u)-1}),\]
with the convention $V^{(n)}(w_n^{(n)})=S_n$, where $u|_k$ is defined as the last $k$ items of $u$, i.e.
\[u|_k:=(u^{(|u|-k+1)},\ldots,u^{(|u|)}).\]
Then for any $|u|=n$, we have
\begin{enumerate}[1)]
\item If $u=w^{(n)}_n$, $V^{(n)}(u)-V^{(n)}(w^{(n)}_n)=0$;
\item If $u\neq w^{(n)}_n$, $V^{(n)}(u)-V^{(n)}(w^{(n)}_n)=b_{n-r(u)}(u^{(r(u)+1)})+V^{(u^{(r(u)+1)},n-r(u))}(u|_{n-r(u)-1})-\sum_{i=1}^{n-r(u)}b_i(w^{(i)}),$
\end{enumerate}
which implies that
\[D_n^\theta=\sum_{|u|=n}\delta_{V^{(n)}(u)-V^{(n)}(w^{(n)}_n)}.\]
 Observe that $(((b_1(i):i\geq 1),w^{(1)}),\ldots,((b_n(i):i\geq 1),w^{(n)}))$ is distributed as $(((b_n(i):i\geq 1),w^{(n)}),\ldots,((b_1(i):i\geq 1),w^{(1)}))$. Therefore, we deduce that $((V^{(n)}(u):|u|=n),w_n^{(n)})$ is distributed as $((\bar{V}^{(n)}(u):|u|=n),\bar{w}_n)$, where
 \[
 \begin{split}
 \bar{V}^{(n)}(u):&=\sum_{i=n-\bar{r}(u)+1}^{n}b_{n-i+1}(w^{(n-i+1)})+b_{\bar{r}(u)+1}(u^{(\bar{r}(u)+1)})+V^{(u^{(\bar{r}(u)+1)},n-\bar{r}(u))}(u|_{n-\bar{r}(u)-1})\\
 &=S_{\bar{r}(u)}+b_{\bar{r}(u)+1}(u^{(\bar{r}(u)+1)})+V^{(u^{(\bar{r}(u)+1)},n-\bar{r}(u))}(u|_{n-\bar{r}(u)-1}),
 \end{split}
 \]
 $\bar{w}_n:=(w^{(1)},\ldots,w^{(n)})$ and $\bar{r}(u):=r(u,\bar{w}_n)$.
 By its definition, we observe that $((\bar{V}^{(n)}(u):|u|=n),\bar{V}^{(n)}(\bar{w}_n))$ under $\bP$ is distributed as $((V(u):|u|=n),V(\xi_n))$ under $\tilde{\bP}$.
 Thus
\[
\begin{split}
&\tilde{\E}\left(\frac{1}{\mathcal{E}_n^{\xi_n}(\{0\})}e^{-\theta V(\xi_n) +n\psi(\theta)}F(\mathcal{E}_n^{\xi_n})f(V(\xi_n)-\psi'(\theta)n)1_{\{\mathcal{E}_n^{\xi_n}((0,\infty))=0\}}\right)\\
=&\E\left(e^{-\theta S_n+n\psi(\theta)}F(D_n^\theta)f(S_n-n\psi'(\theta))\frac{1}{D_n^\theta(\{0\})}1_{\{D_n^\theta((0,\infty))=0\}}\right)\\
=&e^{n(\psi(\theta)-\theta\psi'(\theta))}\E\left(e^{-\theta S_n+n\theta\psi'(\theta)}F(D_n^\theta)f(S_n-n\psi'(\theta))\frac{1}{D_n^\theta(\{0\})}1_{\{D_n^\theta((0,\infty))=0\}}\right),
\end{split}
\]
which completes the proof.
\end{proof}

Define the increasing limit of $D^\theta_n$ as $n\to\infty$ by
\[D_\infty^\theta:=\delta_0+\sum_{k=1}^{\infty}\sum_{i\neq w^{(k)}}\sum_{|u|=k-1}\delta_{b_k(i)+V^{(i,k)}(u)-S_{k}}.\]
Using the lemma below, we will prove that $D_\infty^\theta$ is a well-defined point process.
\begin{lemma}\label{well}
 Let $\theta>0$ such that (\ref{as1}) and (\ref{as3}) hold. For $\vep \in(0,\frac{1}{2}(\psi'(\theta)-\frac{\psi(\theta)}{\theta}))$, we define
\[A_{n,\vep}:=\left\{\max_{\ell\geq n}\frac{\max_{k\geq 1}M^{(k,\ell)}-S_{\ell}}{\ell}<-\vep\right\},\]
where $M^{(k,\ell)}:=b_\ell(k)+\max_{|u|=\ell-1}\{V^{(k,\ell)}(u)\}$.
Then as $n\to\infty$, $\bP(A_{n,\vep}^c)\to0.$
\end{lemma}
\begin{proof}
We denote by $\mathcal{Y}$ the sigma-field generated by $(b_\ell(k),w^{(\ell)},\ell\geq 1,k\geq 1)$. By the independence of $\mathcal{Y}$ and $(V^{(i,k)}(u):u\in\mathbb{T},i\geq 1,k\geq 0)$, for $\ell\geq n$, we have that almost surely
\[
\begin{split}
\bP(A_{n,\vep}^c|\mathcal{Y})
\leq &\sum_{\ell\geq n}\E(1_{\{\max_{k\geq 1}M^{(k,\ell)}-S_{\ell}\geq -\vep\ell\}}|\mathcal{Y})\\
\leq &\sum_{\ell\geq n}\sum_{k=1}^\infty e^{\theta b_\ell(k)-\theta S_\ell+\theta\vep\ell}\E(e^{\theta\max_{|u|=\ell-1}V^{(k,\ell)}(u)}).
\end{split}
\]
Using that $\E(e^{\theta\max_{|u|=\ell-1}V^{(k,\ell)}(u)})\leq \E(\sum_{|u|=\ell-1}e^{\theta V(u)})$, we have
\begin{equation}\label{conexp}
\bP(A_{n,\vep}^c|\mathcal{Y})\leq \sum_{\ell\geq n}e^{(\ell-1)\psi(\theta)+\vep\theta\ell}e^{-\theta S_\ell}\sum_{k=1}^\infty e^{\theta b_\ell(k)}.
\end{equation}
Observe that
\[\bP\left(\sum_{k=1}^\infty e^{\theta b_\ell(k)}\geq e^{\vep\theta\ell}\right)=\bP\left(\log_+\left(\sum_{k=1}^\infty e^{\theta b_\ell(k)}\right)\geq \vep\theta\ell\right).\]
By assumption (\ref{as3}) and the definition of $b_\ell(k)$, we have
\[\E\left(\log_+\left(\sum_{k=1}^\infty e^{\theta b_\ell(k)}\right)\right)=\E\left(\sum_{|u|=1}e^{\theta V(u)}\log_+\left(\sum_{|u|=1}e^{\theta V(u)}\right)\right)<\infty,\]
which implies that
\[\sum_{\ell\geq 1}\bP\left(\sum_{k=1}^\infty e^{\theta b_\ell(k)}\geq e^{\vep\theta\ell}\right)=\sum_{\ell\geq 1}\bP\left(\log_+\left(\sum_{k=1}^\infty e^{\theta b_\ell(k)}\right)\geq \vep\theta\ell\right)<\infty.\]
By Borel-Cantelli Lemma, we have
\begin{equation}\label{cb}
\sum_{k=1}^\infty e^{\theta b_\ell(k)}<e^{\vep\theta\ell}
\end{equation}
almost surely for $\ell$ large enough.
On the other hand, by the strong law of large numbers, we have
\[\lim_{n\to\infty}\frac{S_n}{n}=\E(S_1)=\psi'(\theta),\quad\text{a.s.},\]
which, by assumption that $\theta\psi'(\theta)>\psi(\theta)$, implies that almost surely,
\begin{equation}\label{slln}
\lim_{n\to\infty}\frac{\psi(\theta)n-\theta S_n+2\vep n}{n}= \psi(\theta)-\theta\psi'(\theta)+2\vep<0.
\end{equation}
Combine (\ref{conexp}) with (\ref{cb}) and (\ref{slln}), we deduce that almost surely,
\[\lim_{n\to\infty}\bP(A_{n,\vep}^c|\mathcal{Y})=0,\]
which, by dominated convergence theorem, yields that $\lim_{n\to\infty}\bP(A_{n,\vep}^c)=0.$
\end{proof}
Now, we prove that $D_\infty^\theta$ is a well-defined Radon measure.
\begin{lemma}\label{well plus}
Let $\theta>0$ such that (\ref{as1}) and (\ref{as3}) hold. For any $y\in\R$,
$D_\infty^\theta((y,\infty))<\infty$ almost surely, and $D_n^\theta$ converges to $D_\infty^\theta$ almost surely for the topology of the vague convergence.
\end{lemma}
\begin{proof}
Fix $\vep\in(0,\frac{1}{2}(\psi'(\theta)-\frac{\psi(\theta)}{\theta}))$.
Recall that $M^{(k,\ell)}=b_\ell(k)+\max_{|u|=\ell-1}V^{(k,\ell)}(u)$. Observe that
\[A_{n,\vep}\subset\{\lim_{n\to\infty}\max_{\ell\geq n}\{\max_{k\geq 1}M^{(k,\ell)}-S_{\ell}\}=-\infty\},\]
which, by Lemma \ref{well}, implies that
\begin{equation}\label{almostsurely}
\bP(\lim_{n\to\infty}\max_{\ell\geq n}\{\max_{k\geq 1}M^{(k,\ell)}-S_{\ell}\}=-\infty)=1.
\end{equation}
Fix $y\in \R$, by (\ref{almostsurely}), in order to prove $\bP(D_\infty^\theta((y,\infty))<\infty)=1$, it is sufficient to prove that for any $n\geq 1$,
\[\bP(D_n^\theta((y,\infty))<\infty)=1.\]
Observe that for any $n\geq 1$, by the independence between $(V^{(k,\ell)}(u):u\in\T)_{k\geq 1,\ell\geq 1}$ and $\mathcal{Y}$, which is defined in Lemma \ref{well}, we have
\[
\bP(D_n^\theta((y,\infty))<\infty|\mathcal{Y})=\bP\left.\left(\sum_{\ell=1}^n\sum_{k\geq 1}\sum_{|u|=k-1}1_\{y_{k,\ell}+V^{(k,\ell)}(u)-y_\ell>y\}<\infty\right)\right|_{y_{k,\ell}=b_\ell(k),y_\ell=S_\ell}.
\]
For any $y_{k,\ell},y_\ell\in\R$, by Markov inequality and many-to-one formula, we have
\[
\E\left(\sum_{\ell=1}^n\sum_{k\geq 1}\sum_{|u|=k-1}1_\{y_{k,\ell}+V^{(k,\ell)}(u)-y_\ell>y\}\right)\leq \sum_{\ell=1}^n\sum_{k\geq 1}e^{(k-1)\psi(\theta)}e^{\theta( y_{k,\ell}-y_\ell-y)}.
\]
According to the assumption that $\psi(\theta)<\infty$, we know that for any $\ell\geq 1$,
\begin{equation}\label{finite}
\bP\left(\sum_{k\geq 1}e^{\theta b_\ell(k)}<\infty\right)=\E\left(\sum_{|u|=1}e^{\theta V(u)-\psi(\theta)}1_{\{\sum_{|u|=1}e^{\theta V(u)}<\infty\}}\right)=1,
\end{equation}
which implies that almost surely
\[\bP(D_n^\theta((y,\infty))<\infty|\mathcal{Y})=1.\]
Thus, we obtain that $\bP(D_n^\theta((y,\infty))<\infty)=1.$

 For any $\phi\in C_c(\R)$, where $C_c(\R)$ is the set of continuous function with compact support on $\R$, by (\ref{almostsurely}), there exists a random time $N$ satisfying for any $n>N$,
 \[\max_{\ell\geq n}\{\max_{k\geq 1}M^{(k,\ell)}-S_{\ell}\}< \text{supp}(\phi)\]
where $\text{supp}(\phi)$ is the closure of $\{x\in\R:|\phi(x)|>0\}.$
Therefore $D_\infty^\theta(\phi)=D_N^\theta(\phi)=D_n^\theta(\phi)$ for any $n>N$, which yields
\[\lim_{n\to\infty}D_n^\theta(\phi)=D_\infty^\theta(\phi),\quad \text{a.s.,}\]
where $\mu(\phi):=\int_{\R}\phi(x)\mathrm{d}\mu(x)$ for a measure $\mu$ on $\R$.
\end{proof}
\subsection{Applications of local limit theorem of a non-lattice random walk}\label{llt}
 In this section, we make use of the local limit theorem in \cite{stone} to give estimates on the law of a non-lattice random walk, which is important for studying the asymptotic behavior of $D_n^\theta$ and $S_n$ in the Section 3.
\begin{lemma}\label{rw}
Consider a non-lattice random walk $(S_n)_{n\geq 0}$ with $S_0=0$. Assume that $S_1$ has mean 0 and finite variance $\sigma^2$.
Then for any non-negative direct Riemann integrable function $g:\R\to[0,\infty)$,
we have
\[\lim_{n\to\infty}\sup_{|y|\leq \gamma_n}\left|\sqrt{n}e^{\frac{y^2}{2\sigma^2n}}\E(g(S_n+y))-\frac{1}{\sqrt{2\pi}\sigma}\int_{\R}g(x)\mathrm{d}x\right|=0,\]
where $(\gamma_n)_{n\geq 0}$ is a non-negative sequence such that $\gamma_n=O(\sqrt{n})$.
\end{lemma}
\begin{proof}
For any $M>0$, $h>0$ and $s\in\R$, we have
\[
\sqrt{n}\E(g(S_n+s))\leq \sum_{ k\in\Z}\sup_{y\in[kh,(k+1)h)}g(y)\sqrt{n}\bP(S_n\in[kh-s,(k+1)h-s)).
\]
According to Corollary 1 in \cite{stone}, for $\vep_1>0$ and $h>0$, for $n$ large enough, uniformly in $s\in\R$, we have
\begin{equation}\label{upper}
\begin{split}
 \sqrt{n}\E(g(S_n+s))
&\leq \sum_{-\frac{M}{h}\leq k\leq \frac{M}{h}}\sup_{x\in[kh,(k+1)h)}g(x) \frac{1}{\sqrt{2\pi}\sigma}he^{-\frac{\left(\frac{(2k+1)h}{2}-s\right)^2}{2n\sigma^2}}\\
&+\sum_{k\leq -\frac{M}{h}}\sup_{x\in[kh,(k+1)h)}g(x) \frac{1}{\sqrt{2\pi}\sigma}h\\
&+\sum_{k\geq \frac{M}{h}}\sup_{x\in[kh,(k+1)h)}g(x) \frac{1}{\sqrt{2\pi}\sigma}h+\sum_{k\in \Z}\sup_{x\in[kh,(k+1)h)}g(x)h\vep_1.
\end{split}
\end{equation}
In particular, by (\ref{upper}), we have
\[\sup_{n\geq 1}\sqrt{n}\E(g(S_n+s))\leq \left(\frac{1}{\sqrt{2\pi}\sigma}+\vep_1\right)\sum_{k\in \Z}\sup_{x\in[kh,(k+1)h)}g(x)<\infty,\]
for $h$ small enough as $g$ is direct Riemann integrable.

On the other hand, we know that  for any $M>0$, $h>0$ and $s\in\R$,
\[
\sqrt{n}\E(g(S_n+s))\geq \sum_{-\frac{M}{h}\leq k\leq \frac{M}{h}}\inf_{x\in[kh,(k+1)h)}g(x)\sqrt{n}\bP(S_n\in[kh-s,(k+1)h-s))
\]
Similarly to above, for all $\vep_1>0$ and $h>0$, for $n$ large enough, uniformly in $s\in\R$,
\begin{equation}\label{lower}
\begin{split}
\sqrt{n}\E(g(S_n+s))
\geq &\sum_{-\frac{M}{h}\leq k\leq \frac{M}{h}}\inf_{x\in[kh,(k+1)h)}g(x)\frac{1}{\sqrt{2\pi}\sigma}\delta_ne^{-\frac{\left(\frac{(2k+1)h}{2}-s\right)^2}{2n\sigma^2}}\\
\quad\quad &-\sum_{-\frac{M}{h}\leq k\leq \frac{M}{h}}\inf_{x\in[kh,(k+1)h)}g(x)h\vep_1.
\end{split}
\end{equation}
Using that $\gamma_n=O(\sqrt n)$, we have
\[\lim_{n\to\infty}\sup\left\{\left|e^{\frac{y^2}{2\sigma^2 n}}e^{-\frac{\left(\frac{(2k+1)h}{2}-y\right)^2}{2n\sigma^2}}-1\right|:|y|\leq \gamma_n,-\frac{M}{h}\leq k\leq \frac{M}{h}\right\}= 0,\]
which, combined with (\ref{upper}) and (\ref{lower}), yields that for $h>0$,
\[
\begin{split}
&\limsup_{n\to\infty}\sup_{|y|\leq \gamma_n}\left|\sqrt{n}e^{\frac{y^2}{2\sigma^2 n}}\E(g(S_n+y))-\frac{1}{\sqrt{2\pi}\sigma}\int_{\R}g(x)\mathrm{d}x\right|\\
\leq&\frac{1}{\sqrt{2\pi}\sigma}\left|\sum_{-\frac{M}{h}\leq k\leq \frac{M}{h}}\sup_{x\in[kh,(k+1)h)}g(x) h-\int_{\R}g(x)\mathrm{d}x\right|+C\sum_{k\in \Z}\sup_{x\in[kh,(k+1)h)}g(x)h\vep_1\\
&+\frac{1}{\sqrt{2\pi}\sigma}\left|\int_{\R}g(x)\mathrm{d}x-\sum_{-\frac{M}{h}\leq k\leq \frac{M}{h}}\inf_{x\in[kh,(k+1)h)}g(x)h\right|\\
&+C\sum_{k\geq \frac{M}{h}}\sup_{x\in[kh,(k+1)h)}g(x)h+C\sum_{k\leq -\frac{M}{h}}\sup_{x\in[kh,(k+1)h)}g(x)h.
\end{split}
\]
As $g$ is directly Riemann integral, letting $h\to 0+$, we have
\[
\begin{split}
&\limsup_{n\to\infty}\sup_{|y|\leq \gamma_n}\left|\sqrt{n}e^{\frac{y^2}{2\sigma^2 n}}\E(g(S_n+y))-\frac{1}{\sqrt{2\pi}\sigma}\int_{\R}g(x)\mathrm{d}x\right|\\
\leq &C\left(\int_M^\infty g(x)\mathrm{d}x+\int_{-\infty}^{-M}g(x)\mathrm{d}x+\vep_1\int_{\R}g(x)\mathrm{d}x\right).
\end{split}
\]
Letting $M\to\infty$ and $\vep_1\to 0+$, we get
\[\lim_{n\to\infty}\sup_{|y|\leq \gamma_n}\left|\sqrt{n}e^{\frac{y^2}{2\sigma^2 n}}\E(g(S_n+y))-\frac{1}{\sqrt{2\pi}\sigma}\int_{\R}g(x)\mathrm{d}x\right|=0.\qedhere\]
\end{proof}

\section{Limit estimates for the auxiliary point process}\label{app}
The main aim of this section is to prove the forthcoming Proposition \ref{main}, which is an essential tool for the proof of Theorems \ref{th2} and \ref{th1}.
\begin{proposition}\label{main}
Assume (\ref{as1}), (\ref{as2}), (\ref{as3}) and (\ref{as4}). Then for any non-negative function $\phi\in C_c(\R)$, as $n\to\infty$, we have
\[\sqrt{2\pi n}\theta \sigma e^{\frac{y^2}{2\sigma^2n}}e^{\theta y}e^{n(\theta\psi'(\theta)-\psi(\theta))}\E(e^{-\mathcal{E}_n(\phi)}1_{\{M_n\geq n\psi'(\theta)+y\}})\to\E\left[\frac{1}{D_\infty^\theta(\{0\})}e^{-D_\infty^\theta(\phi)}1_{\{D_\infty^{\theta}((0,\infty))=0\}}\right]\]
 uniformly in $|y|\leq a_n$, where $(a_n)_{n\geq 0}$ is the sequence defined in Theorem \ref{th2}.
\end{proposition}

This proposition will allow us to study the asymptotic behavior of the joint law of the extremal process $\mathcal{E}_n$ and the maximal displacement $M_n$.

Let $(c_n)_{n\geq 0}$ be a non-negative integer sequence, satisfying
\begin{equation}\label{cn}
c_n<n,\quad\lim_{n\to\infty}c_n=\infty\quad\text{and}\quad\limsup_{n\to\infty}\frac{c_n^2}{n}<\infty.
\end{equation}
For $\vep>0$, recall that
\begin{equation}\label{An}
A_{c_n,\vep}:=\left\{\max_{\ell\geq c_n}\frac{\max_{k\geq 1}\{b_\ell(k)+\max_{|u|=\ell-1}V^{(k,\ell)}(u)\}-S_\ell}{\ell}<-\vep\right\}.
\end{equation}
According to Lemma \ref{lemD}, for any non-negative function $\phi\in C_c(\R)$ and $y\in \R$,
 we have
\[
\begin{split}
&\E(e^{-\mathcal{E}_n(\phi)}1_{\{M_n-n\psi'(\theta)\geq y\}})\\
=&e^{n(\psi(\theta)-\theta\psi'(\theta))}\E\left[\frac{1}{D^\theta_n(\{0\})}e^{-\theta (S_n-n \psi'(\theta))}e^{-D^\theta_n(\phi)}1_{\{S_n-\psi'(\theta) n\geq y\}}1_{\{D^\theta_n((0,\infty))=0\}}\right].
\end{split}
\]
Set
\begin{equation}\label{h}
h_n^\theta(\phi):=\frac{1}{D^\theta_n(\{0\})}e^{-D^\theta_n(\phi)}1_{\{D^\theta_n((0,\infty))=0\}},
\end{equation}
and let $x_\phi:=\inf \text{supp}(\phi)$. For $n$ such that $-\vep c_n<x_\phi$, we observe that $h^\theta_n(\phi)=h^\theta_{c_n}(\phi)$ on $A_{c_n,\vep}$, therefore
\begin{equation}\label{split}
\begin{split}
&\E\left[\frac{1}{D^\theta_n(\{0\})}e^{-\theta (S_n-n \psi'(\theta))}e^{-D^\theta_n(\phi)}1_{\{S_n-\psi'(\theta) n\geq y\}}1_{\{D^\theta_n((0,\infty))=0\}}\right]\\
=&\E[h_n^\theta(\phi) e^{-\theta(S_n-n\psi'(\theta))}1_{\{S_n-n\psi'(\theta)\geq y\}}1_{A_{c_n,\vep}}]+\E[h_n^\theta(\phi) e^{-\theta(S_n-n\psi'(\theta))}1_{\{S_n-n\psi'(\theta)\geq y\}}1_{A_{c_n,\vep}^c}]\\
=&I_1(n,y)+I_2(n,y,\vep)+I_3(n,y,\vep),
\end{split}
\end{equation}
where
\[
\begin{split}
&I_1(n,y):=\E[h_{c_n}^\theta(\phi) e^{-\theta(S_n-n\psi'(\theta))}1_{\{S_n-n\psi'(\theta)\geq y\}}], \\
&I_2(n,y,\vep):=-\E[h_{c_n}^\theta(\phi) e^{-\theta(S_n-n\psi'(\theta))}1_{\{S_n-n\psi'(\theta)\geq y\}}1_{A_{c_n,\vep}^c}],\\
&I_3(n,y,\vep):=\E[h_{n}^\theta(\phi) e^{-\theta(S_n-n\psi'(\theta))}1_{\{S_n-n\psi'(\theta)\geq y\}}1_{A_{c_n,\vep}^c}].
\end{split}
\]

Using Lemma \ref{rw}, it is straightforward to get an equivalent of $I_1(n,y)$ and the decay rate of $I_2(n,y,\vep)$ and $I_3(n,y,\vep)$ with respect to $n$. We prove these two results in turn in the following lemmas.
\begin{lemma}\label{part1}
Assume (\ref{as1}), (\ref{as2}) and (\ref{as4}).
Then for any non-negative function $\phi\in C_c(\R)$,
\[\lim_{n\to\infty}\sup_{|y|\leq a_n}\left|\sqrt{n}e^{\frac{y^2}{2\sigma^2n}}e^{\theta y}I_1(n,y)-\frac{1}{\sqrt{2\pi}\sigma\theta}\E\left(\frac{1}{D_\infty^\theta(\{0\})}e^{-D_\infty^\theta(\phi)}1_{\{D_\infty^\theta((0,\infty))=0\}}\right)\right|=0,\]
where $(a_n)_{n\geq 0}$ is the sequence defined in Theorem \ref{th2}.
\end{lemma}
\begin{proof}
Observe that, by Markov property,
\[e^{\theta y}I_1(n,y)=\E(h_{c_n}^\theta(\phi)g_n(S_{c_n}-c_n\psi'(\theta)-y)),\]
where for $s\in\R$, we write
\[g_n(s):=\E(e^{-\theta(s+ S_{n-c_n}-(n-c_n) \psi'(\theta))}1_{\{S_{n-c_n}-(n-c_n) \psi'(\theta)+s\geq 0\}}).\]
By the assumptions on $(c_n)_{n\geq 0}$, it is straightforward to see that as $n\to\infty$,
\[\frac{n}{c_n}\to\infty\quad \text{and}\quad\frac{S_{c_n}-c_n\psi'(\theta)}{\sqrt{n-c_n}}\to 0,\quad \text{a.s.,}\]
by strong law of large numbers. Moreover, as $a_n=O(\sqrt{n-c_n})$,
by Lemma \ref{rw}, we get that as $n\to\infty$,
\[\sup_{|y|\leq a_n}\left|\sqrt{n}e^{\frac{y^2}{2\sigma^2n}}g_n(S_{c_n}-c_n\psi'(\theta)-y)-\frac{1}{\sqrt{2\pi}\sigma\theta}\right|\to 0,\quad \text{a.s.}\]
By Lemma \ref{well plus} and the monotonicity of $D_n^\theta$, we can get
\[\lim_{n\to\infty}h_{c_n}^\theta(\phi)= h_\infty^\theta(\phi),\quad\text{a.s.},\]
where $h_\infty^\theta(\phi):= \frac{1}{D_\infty^\theta(\{0\})}e^{-D_\infty^\theta(\phi)}1_{\{D_\infty^\theta((0,\infty))=0\}}$.
Then
\[
\begin{split}
&\sup_{|y|\leq a_n}\left|\sqrt{n}e^{\frac{y^2}{2\sigma^2 n}}e^{\theta y}I_1(n,\theta)-\frac{1}{\sqrt {2\pi}\sigma \theta}\E(h_\infty^\theta(\phi))\right|\\
\leq&\E\left(h_\infty^\theta(\phi)\sup_{|y|\leq a_n}\left|\sqrt{n}e^{\frac{y^2}{2\sigma^2 n}}g_n(S_{c_n}-\psi'(\theta)c_n-y)-\frac{1}{\sqrt {2\pi}\sigma\theta}\right|\right)\\
&+\E\left(\left|h_{c_n}^\theta(\phi)-h_\infty^\theta(\phi)\right|\sqrt{n}\sup_{|y|\leq a_n}e^{\frac{y^2}{2\sigma^2 n}}g_n(S_{c_n}-\psi'(\theta)c_n-y)\right).
\end{split}
\]
By (\ref{upper}), there exists $C>0$ and $N>0$ such that for any $n>N$,
\[\sup_{s\in\R}\sqrt{n}g_n(s)<C.\]
Because $a_n=O(\sqrt n)$, we know that there exists $C>0$ such that for any $n\geq 1$,
\begin{equation}\label{bound}
\sup_{|y|\leq a_n}e^{\frac{y^2}{2\sigma^2 n}}<C.
\end{equation}
Besides, using that for any $n\geq 1$, $h_n^\theta(\phi)\leq 1,$
by dominated convergence theorem, we conclude that,
\[\lim_{n\to\infty}\sup_{|y|\leq a_n}\left|\sqrt{n}e^{\frac{y^2}{2\sigma^2 n}}e^{\theta y}I_1(n,y)-\frac{1}{\sqrt {2\pi}\sigma\theta}\E\left(\frac{1}{D_\infty^\theta(\{0\})}e^{-D_\infty^\theta(\phi)}1_{\{D_\infty^\theta((0,\infty))=0\}}\right)\right|= 0.\qedhere\]
\end{proof}
Observe that
\[e^{\theta y}\E(|I_2(n,y,\vep)|+|I_3(n,y,\vep)|)\leq 2\E(e^{-\theta(S_n-n\psi'(\theta)-y)}1_{\{S_n\geq n\psi'(\theta)+y\}}1_{A_{c_n,\vep}^c}),\]
and (\ref{bound}), so it is sufficient to prove the right-hand side is $o(n^{-\frac{1}{2}})$ uniformly in $|y|\leq a_n$.

\begin{lemma}\label{part2}
Assume (\ref{as1}), (\ref{as2}), (\ref{as3}) and (\ref{as4}).
There exists $\vep>0$ such that
\[\lim_{n\to\infty}\sqrt{n}\sup_{|y|\leq a_n}\E(e^{-\theta(S_n-n\psi'(\theta)-y)}1_{\{S_n\geq n\psi'(\theta)+y\}}1_{A_{c_n,\vep}^c})= 0,\]
where $(a_n)_{n\geq 0}$ is the sequence defined in Theorem \ref{th2}.
\end{lemma}
\begin{proof}
For the simplification of notation, we use $x$ instead of $\psi'(\theta)$ in the proof.
Observe that for $n\in\N_+$ and $|y|\leq a_n,$
\[
\begin{split}
&\E(e^{-\theta(S_n-nx-y)}1_{\{S_n\geq nx+y\}}1_{A_{c_n,\vep}^c})\\
\leq &\sum_{\ell\geq  c_n}\E(e^{-\theta(S_n-nx-y)}1_{\{S_n\geq nx+y\}};\max_{k\geq 1}\{b_\ell(k)+\max_{|u|=\ell-1}\{V^{(k,\ell)}(u)\}\}-S_{\ell}\geq -\vep\ell)
\end{split}
\]
Recall that $M^{(k,\ell)}=b_\ell(k)+\max_{|u|=\ell-1}\{V^{(k,\ell)}(u)\}$. Then for $\ell\geq c_n,$
\[
\begin{split}
&\E(e^{-\theta(S_n-nx-y)}1_{\{S_n\geq nx+y\}}1_{\{\max_{k\geq 1}M^{(k,\ell)}-S_{\ell}\geq -\vep\ell\}})\\
\leq &\E(e^{-\theta(S_n-nx-y)}1_{\{S_n\geq nx+y\}}1_{\{\max_{k\geq 1}M^{(k,\ell)}> (x-2\vep)\ell\}})+\E(e^{-\theta(S_n-nx-y)}1_{\{S_n\geq nx+y\}}1_{\{S_{\ell}<(x-\vep)\ell\}})\\
\leq &\E(e^{-\theta(S_n-nx-y)}1_{\{\max_{k\geq 1}M^{(k,\ell)}> (x-2\vep)\ell, \sum_{k\geq 1}e^{\theta b_\ell(k)}\leq e^{c\ell}\}})\\
&+\E(e^{-\theta(S_n-nx-y)}1_{\{S_n\geq nx+y\}}1_{\{\sum_{k\geq 1}e^{\theta b_\ell(k)}>e^{c\ell}\}})+\E(e^{-\theta(S_n-nx-y)}1_{\{S_n\geq nx+y\}}1_{\{S_{\ell}< (x-\vep)\ell\}}),
\end{split}
\]
where $c$ is a positive constant whose value will be fixed later.

Next, we prove that as $n\to\infty,$
\begin{equation}\label{eqR}
\sqrt{n}\sup_{|y|\leq a_n}\sum_{\ell\geq c_n}\E(e^{-\theta(S_n-nx-y)}1_{\{S_n\geq nx+y\}};S_{\ell}< (x-\vep)\ell)\to 0.
\end{equation}
Recall that $c_n/n\to0 $, so $c_n<n/2$ for $n$ large enough.
We first consider $\ell\in[c_n,\frac{n}{2}]$. By (\ref{upper}), for $y\in\R$ and $n$ large enough, we have
\begin{equation}\label{eq1}
\begin{split}
&\E(e^{-\theta(S_n-nx-y)}1_{\{S_n\geq nx+y\}}1_{\{S_{\ell}< (x-\vep)\ell\}})\\
=&\E(1_{\{S_{\ell}<(x-\vep)\ell\}}\E(e^{-\theta(S_{n-\ell}+s-y-nx)}1_{\{S_{n-\ell}+s\geq nx+y\}})|_{s=S_{\ell}})\\
\leq &C\frac{1}{\sqrt {n-\ell}}\bP(S_{\ell}<(x-\vep)\ell).
\end{split}
\end{equation}
According to the inequality (1.56) in \cite{nagaev}, setting $Y_i:=x-b_i(w^{(i)})$, we have
\[
\bP(S_{\ell}< (x-\vep)\ell)=\bP\left(\sum_{i=1}^{\ell}Y_i> \vep \ell\right)\leq \ell \bP\left(Y_1\geq \frac{1}{2}\vep \ell\right)+(2e\sigma^2)^2\vep^{-4}\ell^{-2}.
\]
As $\E Y_1^2$ is finite, we have $\sum_{\ell\geq 1}\ell \bP(Y_1\geq \frac{1}{2}\vep \ell)<\infty$. Thus, we conclude that
\[\lim_{n\to\infty}\sqrt{n}\sup_{|y|\leq a_n}\sum_{\ell\in[c_n,\frac{n}{2}]}\E(e^{-\theta( S_n-nx-y)}1_{\{S_n\geq nx+y\}}1_{\{S_{\ell}<(x-\vep)\ell\}})=0.\]
Next we consider $\ell \in [\frac{n}{2}, n]$. Because
\[\{S_n\geq nx+y,S_\ell<(x-\vep)\ell\}\subset\{S_n\geq nx+y,S_n-S_\ell>(n-\ell)x+y+\vep\ell\},\]
we have
\[\E(e^{-\theta(S_n-nx-y)}1_{\{S_n\geq nx+y\}}1_{\{S_{\ell}< (x-\vep)\ell\}})\leq \E(e^{-\theta(S_n-nx-y)}1_{\{S_n\geq nx+y\}}1_{\{S_n-S_\ell>(n-\ell)x+y+\vep\ell\}}).\]
By the independence between $S_n-S_\ell$ and $S_\ell$ and that $S_n-S_\ell$ is distributed as $S_{n-\ell}$, we have
\[
\begin{split}
&\E(e^{-\theta(S_n-nx-y)}1_{\{S_n\geq nx+y\}}1_{\{S_n-S_\ell>(n-\ell)x+y+\vep\ell\}})\\
=&\E(1_{\{S_{n-\ell}>(n-\ell)x+y+\vep\ell\}}\E(e^{-\theta(s+S_\ell-nx-y)}1_{\{s+S_{\ell}\geq nx+y\}})|_{s=S_{n-\ell}}),
\end{split}
\]
which, with the same arguments as in (\ref{eq1}), yields that for $n$ large enough, any $|y|\leq a_n$ and $\ell\in[\frac{n}{2},n]$, we have
\[
\begin{split}
&\E(1_{\{S_{n-\ell}>(n-\ell)x+y+\vep\ell\}}\E(e^{-\theta(s+S_\ell-nx-y)}1_{\{s+S_{\ell}\geq nx+y\}})|_{s=S_{n-\ell}})\\
\leq& \frac{1}{\sqrt{\ell}}C\bP(S_{n-\ell}>(n-\ell)x+y+\vep\ell)\\
\leq &\frac{1}{\sqrt{\ell}}C\left[\ell\bP\left(S_1-x\geq \frac{1}{2}\vep \ell\right)+\left(\frac{2(n-\ell)\sigma^2}{(\vep\ell+y)\vep\ell}\right)^{\frac{2(\vep\ell+y) }{\vep\ell}}e^{\frac{2(\vep\ell+y ) }{\vep\ell}}\right].
\end{split}
\]
Since that $n-\ell\leq\ell$ and $a_n=O(\sqrt{n})$,
for $n$ large enough,
\[
\sup_{|y|\leq a_n}\left\{\ell\bP\left(S_1-x\geq \frac{1}{2}\vep \ell\right)+\left(\frac{2(n-\ell)\sigma^2}{(\vep\ell+y)\vep\ell}\right)^{\frac{2(\vep\ell+y) }{\vep\ell}}e^{\frac{2(\vep\ell+y ) }{\vep\ell}}\right\}
\leq \ell\bP\left(S_1-x\geq \frac{1}{2}\vep\ell\right)+C\frac{1}{\ell^2}.
\]
Using that $\E(S_1^2)<\infty$ and $\ell\geq \frac{n}{2}$, we obtain
\[\lim_{n\to\infty}\sqrt{n}\sup_{|y|\leq a_n}\sum_{\ell\in[\frac{n}{2},n]}\E(e^{-\theta (S_n-nx-y)}1_{\{S_n\geq nx+y\}}1_{\{S_{\ell}< (x-\vep)\ell\}})=0.\]
With the same arguments as in the case of $\ell\in[\frac{n}{2},n]$, there exists $C>0$ such that for any $\ell>n$,
\[
\begin{split}
&\E(e^{-\theta(S_n-nx-y)}1_{\{S_n\geq nx+y\}}1_{\{S_\ell<(x-\vep)\ell\}})\\
\leq &\E(e^{-\theta(S_n-nx-y)}1_{\{S_n\geq nx+y\}}1_{\{S_\ell-S_n<x(\ell-n)-y-\vep \ell\}})\\
\leq &\frac{C}{\sqrt{n}}\left[\ell\bP\left(x-S_1\geq \frac{1}{2}\vep \ell\right)+\left(\frac{2(\ell-n)\sigma^2}{(\vep\ell+y)\vep\ell}\right)^{\frac{2(\vep\ell+y) }{\vep\ell}}e^{\frac{2(\vep\ell+y ) }{\vep\ell}}\right].
\end{split}
\]
Thus, we obtain that
\[\lim_{n\to\infty}\sqrt{n}\sup_{|y|\leq a_n}\sum_{\ell> n}\E(e^{-\theta (S_n-nx-y)}1_{\{S_n\geq nx+y\}}1_{\{S_{\ell}< (x-\vep)\ell\}})=0,\]
which, combined with results above, implies (\ref{eqR}).

Similarly to the arguments as in (\ref{eq1}), for $n$ large enough and $|y|\leq a_n$,
\[
\begin{split}
&\E(e^{-\theta(S_n-nx-y)}1_{\{S_n\geq nx+y\}}1_{\{\sum_{k\geq 1}e^{\theta b_\ell(k)}>e^{c\ell}\}})\\
=&\E(1_{\{\sum_{k\geq 1}e^{\theta b_1(k)}>e^{c\ell}\}}\E(e^{-\theta(S_{n-1}+s-nx-y)}1_{\{S_{n-1}+s\geq nx+y\}})|_{s=S_1})\\
\leq &C\frac{1}{\sqrt{n}}\bP\left(\log_+\left(\sum_{k\geq 1}e^{\theta b_1(k)}\right)>c\ell\right)
\end{split}
\]
By assumption (\ref{as3}) and the definition of $(b_1(k))_{k\geq 1}$, we know that
\[
\E\left(\log_+\left(\sum_{k\geq 1}e^{\theta b_1(k)}\right)\right)=\E\left(\sum_{|u|=1}e^{\theta V(u)-\psi(\theta)}\log_+\left(\sum_{|u|=1}e^{\theta V(u)}\right)\right)< \infty,
\]
which yields that
\[\lim_{n\to\infty}\sqrt{n}\sup_{|y|\leq a_n}\sum_{\ell\geq c_n}\bP(e^{-\theta(S_n-nx-y)}1_{\{S_n\geq nx+y\}}1_{\{\sum_{k\geq 1}e^{\theta b_\ell(k)}>e^{c\ell}\}})=0.\]

Finally, as $S_n-b_\ell(w^{(\ell)})$ is distributed as $S_{n-1}$ and is independent of $(b_\ell(k))_{k\geq 1}$, we obtain
\[
\begin{split}
&\E(e^{-\theta(S_n-nx-y)}1_{\{S_n\geq nx+y\}}1_{\{\max_{k\geq 1}M^{(k,\ell)}> (x-2\vep)\ell, \sum_{k\geq 1}e^{\theta b_\ell(k)}\leq e^{c\ell}\}})\\
=&\E(1_{\{\max_{k\geq 1}M^{(k,\ell)}> (x-2\vep)\ell, \sum_{k\geq 1}e^{\theta b_\ell(k)}\leq e^{c\ell}\}}\E(e^{-\theta(S_{n-1}+s-nx-y)}1_{\{S_{n-1}+s\geq nx+y\}})|_{s=b_\ell( w^{(\ell)})}).
\end{split}
\]
With the same arguments as in (\ref{eq1}), we have
\[
\begin{split}
&\E(1_{\{\max_{k\geq 1}M^{(k,\ell)}> (x-2\vep)\ell, \sum_{k\geq 1}e^{\theta b_\ell(k)}\leq e^{c\ell}\}}\E(e^{-\theta(S_{n-1}+s-nx-y)}1_{\{S_{n-1}+s\geq nx+y\}})|_{s=b_\ell(w^{(\ell)})})\\
\leq &\frac{C}{\sqrt n}\E\left(\sum_{k\geq 1 }e^{\theta M^{(k,\ell)}-\theta (x-2\vep)\ell}1_{\{\sum_{k\geq 1}e^{\theta b_\ell(k)}\leq e^{c\ell}\}}\right)\\
\leq &\frac{C}{\sqrt n}\E\left(\sum_{|u|=\ell-1}e^{\theta V(u)-\theta(x-2\vep)\ell}\right)e^{c\ell}\\
=&\frac{C}{\sqrt{n}}e^{(c+2\vep\theta)\ell}e^{-\ell(\theta x-\psi(\theta))},
\end{split}
\]
where the last equality follows form many-to-one formula, and $c$ and $\vep$ are chosen to satisfy $c+2\vep\theta<\theta x-\psi(\theta)$.
As $n\to\infty$, we get
\[\sqrt{n}\sup_{|y|\leq a_n}\sum_{\ell\geq c_n}\E(e^{-\theta(S_n-nx-y)}1_{\{S_n\geq nx+y\}}1_{\{\max_{k\geq 1}M^{(k,\ell)}> (x-2\vep)\ell, \sum_{k\geq 1}e^{\theta b_\ell(k)}\leq e^{c\ell}\}})\to 0.\]
Thus, we conclude that as $n\to\infty$,
\[\sqrt{n}\sup_{|y|\leq a_n}\E(e^{-\theta(S_n-n\psi'(\theta)-y)}1_{\{S_n\geq n\psi'(\theta)+y\}}1_{A_{c_n,\vep}^c})\to 0.\qedhere\]
\end{proof}
Let
\begin{equation}\label{constantC}
C(\theta):=\E(\frac{1}{D_\infty^\theta(\{0\})}1_{\{D_\infty^\theta((0,\infty))=0\}})
\end{equation}
and we bound its value in the following two lemmas.
\begin{lemma}\label{lem1}
Let $\theta>0$ such that (\ref{as1}) and (\ref{as3}) hold. Then $C(\theta)\in(0,1]$.
\end{lemma}
\begin{proof}
Using that $D_\infty^\theta(\{0\})\geq 1$, it is straightforward to deduce that $C(\theta)\leq 1$. Next, we prove that $C(\theta)>0$.
Fix $\vep\in(0,\frac{1}{2}(\psi'(\theta)-\frac{\psi(\theta)}{\theta}))$ from Lemma \ref{well}.
Recall that $\bP(A_{n,\vep}^c)\to 0.$
Then there exists $N>0$ such that $\bP(A_{N+1,\vep})>0$.
 Define
\[f_N(s):=\bP\left(\max_{\ell\geq {N+1}}\frac{\max_{k\geq 1}\{b_\ell(k)+\max_{|u|=\ell-1}\{V^{(k,\ell)}(u)\}\}-(S_\ell-S_{N})-s}{\ell}<-\vep\right),\quad s\in\R.\]
Observe that
\[
\begin{split}
\E\left(\frac{1}{D_\infty^\theta(\{0\})}1_{\{D_\infty^\theta((0,\infty))=0\}}\right)&\geq \E\left(\frac{1}{D_\infty^\theta(\{0\})}1_{\{D_\infty^\theta((0,\infty))=0\}}1_{A_{N+1,\vep}}\right)\\
&=\E\left(\frac{1}{D_N^\theta(\{0\})}1_{\{D_N^\theta((0,\infty))=0\}}f_N(S_N)\right)\\
&=e^{N(\theta\psi'(\theta)-\psi(\theta))}\E(e^{\theta M_N-N\psi(\theta)}f_N(M_N)1_{\{Z_N(\R)>0\}}),
\end{split}
\]
where the last equation follows from Lemma \ref{lemD}.
Thus, in order to prove Lemma \ref{lem1}, it is sufficient to prove
\[\E(f_N(M_N)1_{\{Z_N(\R)>0\}})>0.\]
It is straightforward to see that $f_N(s)$ is increasing with respect to $s$. Denote
\[\gamma:=\sup\{s\in\R:f_N(s)=0\},\]
with the usual convention $\sup\varnothing=-\infty$.\\
\textbf{Case 1: $\gamma=-\infty$}

Because $\gamma=-\infty$, we know that for any $s\in\R$, $f_N(s)>0$. Hence $f_N(M_N)>0$ on the event $\{Z_N(\R)>0\}$, which yields that
\[\E(f_N(M_N)1_{\{Z_N(\R)>0\}})>0.\]
\textbf{Case 2: $\gamma>-\infty$}

We assert that $\bP(S_N\geq\gamma)>0$. If $\bP(S_N\geq\gamma)=0$, then
$\bP(S_N<\gamma)=1$. We can find $\vep_1>0$ such that
\[\bP(A_{N+1,\vep},S_N<\gamma-\vep_1)>0.\]
However,
\[
\bP(A_{N+1,\vep},S_N<\gamma-\vep_1)=\E(f_N(S_N)1_{\{S_N<\gamma-\vep_1\}})\leq f_N(\gamma-\vep_1)=0,
\]
which yields a contradiction.
Thus, we obtain that $\bP(S_N\geq\gamma)>0$.

By many-to-one formula and the spine decomposition theorem, we know that
\[\E\left(\sum_{|u|=N}e^{\theta V(u)-N\psi(\theta)}1_{\{V(u)\geq \gamma\}}\right)=\bP(S_N\geq \gamma)>0,\]
which implies that
\[\bP(M_N\geq \gamma)>0.\]
If $f_N(\gamma)>0$, then we have
\[\E(f_N(M_N)1_{\{Z_N(\R)>0\}})\geq f_N(\gamma)\bP(M_n\geq \gamma)>0.\]
If $f_N(\gamma)=0$, as $\E(f_N(S_N))>0$, there exists $L>\gamma$ such that $\bP(S_N\geq L)>0$. Similarly to the arguments above, we get
\[\E(f_N(M_N)1_{\{Z_N(\R)>0\}})\geq f_N(L)\bP(M_N\geq L )>0,\]
which completes the proof.
\end{proof}
\begin{lemma}\label{constant}
Let $\theta>0$ such that (\ref{as1}) and (\ref{as3}) hold. Then $C(\theta)=1$ implies $\bP(Z_1(\R)\leq 1)=1$.
\end{lemma}
\begin{proof}
 Because $C(\theta)=1$, we have
\[\E\left(\frac{1}{D_\infty^\theta(\{0\})}1_{\{D_\infty^\theta((0,\infty))=0\}}\right)=1,\]
which, as $D_\infty^\theta(\{0\})\geq 1$, implies that
\[D_\infty^\theta(\{0\})=1\quad\text{and}\quad D_\infty^\theta((0,\infty))=0,\quad \text{a.s.}\]
Note that for any $n\geq 0$, $0\leq D_n^\theta((0,\infty))\leq D_\infty^\theta((0,\infty))$ and $1\leq D_n^\theta(\{0\})\leq D_\infty^\theta(\{0\})$, we obtain that
\[D_n^\theta(\{0\})=1\quad\text{and}\quad D_n^\theta((0,\infty))=0,\quad\text{a.s.},\]
which, by Lemma \ref{lemD}, yields that for any $n\geq 0$,
\[\bP(Z_n(\R)>0)=\E(e^{-\theta S_n+n\psi(\theta)}).\]
By the definition of $S_n$, we have
\[\bP(Z_n(\R)>0)=\E(Z_n(\R)).\]
In particular, taking $n=1$, we get $\bP(Z_1(\R)>0)=\E(Z_1(\R))$, which implies that $\bP(Z_1(\R)\leq 1)=1$.
\end{proof}
Finally, we prove Proposition \ref{main}.
\begin{proof}[Proof of Proposition \ref{main}]
For any non-negative $\phi\in C_c(\R)$ and $|y|\leq a_n$, by Lemma \ref{lemD}, we have
\[
\begin{split}
&\sqrt{n}e^{\theta y}e^{n(\theta\psi'(\theta)-\psi(\theta))}\E(e^{-\mathcal{E}_n(\phi)}1_{\{M_n-n\psi'(\theta)\geq y\}})\\
=&\sqrt{n}\E\left[\frac{1}{D^\theta_n(\{0\})}e^{-\theta (S_n-n \psi'(\theta)-y)}e^{-D^\theta_n(\phi)}1_{\{S_n-\psi'(\theta) n\geq y\}}1_{\{D^\theta_n((0,\infty))=0\}}\right]\\
=&\sqrt{n}\E(h_n^\theta(\phi)e^{-\theta (S_n-n \psi'(\theta)-y)}1_{\{S_n-\psi'(\theta) n\geq y\}})
\end{split}
\]
where $h_n^\theta(\phi)$ is defined in (\ref{h}) and we define
\[h_\infty^\theta(\phi)=\frac{1}{D_\infty^\theta(\{0\})}e^{-D_\infty^\theta(\phi)}1_{\{D_\infty^\theta((0,\infty))=0\}}.\]
By the decomposition of (\ref{split}) and the estimations of Lemma \ref{part1} and \ref{part2}, there exists $\vep>0$ such that as $n\to\infty$,
\[
\begin{split}
&\sup_{|y|\leq a_n}\left|\sqrt{n}e^{\frac{y^2}{2\sigma^2n }}\E(h_n^\theta(\phi)e^{-\theta (S_n-n \psi'(\theta)-y)}1_{\{S_n-\psi'(\theta) n\geq y\}})-\frac{1}{\sqrt{2\pi}\theta\sigma}\E(h_\infty^\theta(\phi))\right|\\
\leq &\sup_{|y|\leq a_n}\left|\sqrt{n}e^{\frac{y^2}{2\sigma^2n }}e^{\theta y}I_1(n,y)-\frac{1}{\sqrt{2\pi}\theta\sigma}\E(h_\infty^\theta(\phi))\right|+\sup_{|y|\leq a_n} \sqrt{n}e^{\frac{y^2}{2\sigma^2n }}e^{\theta y}|I_2(n,y,\vep)+I_3(n,y,\vep)|\\
\to& 0,
\end{split}
\]
which completes the proof.
\end{proof}
\section{Proof of main results}\label{p1th23}
We first prove Theorems \ref{th2} and \ref{th1} in Section \ref{th23}. We then turn to the proof of Proposition \ref{log}, which follows form similar methods as the ones used by Biggins in \cite[Theorem 2]{biggins1}.
\subsection{Proof of Theorems \ref{th2} and \ref{th1}}\label{th23}
Recall that $\mathcal{E}_n=\sum_{|u|=n}\delta_{V(u)-M_n}$. We are now ready to give the proof of the precise upper deviation estimates of $M_n$ and the weak convergence of $(\mathcal{E}_n,M_n-\psi'(\theta)n)$ conditionally on $\{M_n\geq n\psi'(\theta)\}$.

\begin{proof}[Proof of Theorem \ref{th2}]
According to many-to-one formula and the spine decomposition theorem, we know that
\[
\begin{split}
\E(Z_n([n\psi'(\theta),\infty)))=&\E\left(\sum_{|x|=n}1_{\{V(x)\geq n\psi'(\theta)\}}\right)\\
=&e^{n\psi(\theta)-n\theta \psi'(\theta)}\E(e^{-\theta S_n+n\theta \psi'(\theta)}1_{\{S_n\geq n\psi'(\theta)\}})
\end{split}
\]
By Lemma \ref{rw}, we conclude that as $n\to\infty$,
\begin{equation}\label{number}
\sqrt{2\pi}\sigma\sqrt{n}e^{n(\theta \psi'(\theta)-\psi(\theta))}\E(Z_n([n\psi'(\theta),\infty)))\to \int_\R1_{\{x\geq 0\}}e^{-\theta x}\mathrm{d}x=\frac{1}{\theta }.
\end{equation}
Take $\phi\equiv0$ in Proposition \ref{main}, by Lemmas \ref{lem1} and \ref{constant}, we obtain
\[
\begin{split}
\sqrt{2\pi}\theta \sigma \sqrt{n}e^{\frac{y^2}{2\sigma^2 n}}e^{\theta y}e^{n(\theta\psi'(\theta)-\psi(\theta))}\bP(M_n\geq n\psi'(\theta)+y)&\to\E\left(\frac{1}{D_\infty^\theta(\{0\})}1_{\{D_\infty^\theta((0,\infty))=0\}}\right)\\
&=C(\theta)\in(0,1)
\end{split}
\]
uniformly in $|y|\leq a_n$, which, combined with (\ref{number}), yields that
\[\lim_{n\to\infty}\frac{\bP(M_n\geq n\psi'(\theta))}{\E(Z_n([n\psi'(\theta),\infty)))}=C(\theta).\qedhere\]
\end{proof}
The results showed in Theorem \ref{th2} is helpful for the proof of Theorem \ref{th1}.
\begin{proof}[Proof of Theorem \ref{th1}]
It is sufficient to prove that for any non-negative function $\phi\in C_c(\R)$ and $y\geq 0$,
\[\E(e^{-\mathcal{E}_n(\phi)}1_{\{M_n-n\psi'(\theta)\geq y\}}|M_n\geq n\psi'(\theta))\to e^{-\theta y}\frac{\E(\frac{1}{D_\infty^\theta(\{0\})}e^{-D_\infty^\theta(\phi)}1_{\{D_\infty^\theta((0,\infty))=0\}})}{\E(\frac{1}{D_\infty^\theta(\{0\})}1_{\{D_\infty^\theta((0,\infty))=0\}})}.\]
By Proposition \ref{main}, for any non-negative $\phi\in C_c(\R)$ and $y\geq 0$, as $n\to\infty$, we have
\[\sqrt{2\pi}\theta \sigma \sqrt{n}e^{n(\theta\psi'(\theta)-\psi(\theta))}\E (e^{-\mathcal{E}_n(\phi)}1_{\{M_n- n\psi'(\theta)\geq y\}})\to e^{-\theta y}\E(\frac{1}{D_\infty^\theta(\{0\})}e^{-D_\infty^\theta(\phi)}1_{\{D_\infty^\theta((0,\infty))=0\}}).\]
Therefore, using Lemma \ref{lem1} and Theorem \ref{th2}, we obtain that as $n\to\infty$,
\[
\begin{split}
\E(e^{-\mathcal{E}_n(\phi)}1_{\{M_n-n\psi'(\theta)\geq y\}}|M_n\geq n\psi'(\theta))
&=\frac{\E(e^{-\mathcal{E}_n(\phi)}1_{\{M_n-n\psi'(\theta)\geq y\}})}{\bP(M_n\geq n\psi'(\theta))}\\
&\to e^{-\theta y}\frac{\E(\frac{1}{D_\infty^\theta(\{0\})}e^{-D_\infty^\theta(\phi)}1_{\{D_\infty^\theta((0,\infty))=0\}})}{\E(\frac{1}{D_\infty^\theta(\{0\})}1_{\{D_\infty^\theta((0,\infty))=0\}})},
\end{split}
\]
which completes the proof.
\end{proof}
\subsection{Proof of Proposition \ref{log}}\label{p1}
To prove Proposition \ref{log}, we need to estimate the survival probability of a subcritical Galton-Watson process. A Galton-Watson process $(Y_n)_{n\geq 0}$ is defined as follows: we set $Y_0:=1$ and for $n\geq 0$
\[Y_{n+1}:=\sum_{k=1}^{Y_n}\xi_{k,n},\]
where $(\xi_{k,n})_{k,n\geq 0}$ are i.i.d nonnegative integer-valued random variables.
\begin{lemma}\label{lemGW}
Consider a Galton-Watson process $(Y_n)_{n\geq0}$ defined as above. Assume that $m:=\E(Y_1)\in(0,1),$ we have
\[\lim_{n\to\infty}\frac{1}{n}\log \bP(Y_n>0)=\log m.\]
\end{lemma}
\begin{proof}
For $L>0$, we construct a new Galton-Watson process $(Y^{(L)}_n)_{n\geq 0}$ by removing the offspring of each particle which produces more than $L$ children. More precisely, $(Y^{(L)}_n)_{n\geq 0}$ is defined as follows: we set $Y_0^{(L)}:=1$ and
\[Y^{(L)}_{n+1}:=\sum_{k=1}^{Y^{(L)}_n}\xi_{k,n}1_{\{\xi_{k,n}\leq L\}},\quad n\geq 1.\]
By induction, we get $Y^{(L)}_n\leq Y_n$ for any $n\geq 0$.
Let $m^{(L)}:=\E(Y^{(L)}_1)=\E(Y_1 1_{\{Y_1\leq L\}})$. By monotone convergence theorem, we have
\[\lim_{L\to\infty} m^{(L)}=m\in(0,1).\]
As $Y^{(L)}_1=Y_1 1_{\{Y_1\leq L\}}\leq L$, we have $\E(Y^{(L)}_1\log_+(Y^{(L)}_1))<\infty$. According to \cite[P40]{athreya}, as $L$ large enough, we have
\[\lim_{n\to\infty}\frac{1}{n}\log\bP(Y^{(L)}_n>0)=\log m^{(L)}.\]
Therefore
\begin{equation}\label{lowerbound}
\liminf_{n\to\infty}\frac{1}{n}\log \bP(Y_n>0)\geq \lim_{L\to\infty}\lim_{n\to\infty}\frac{1}{n}\log\bP(Y^{(L)}_n>0)=\log m.
\end{equation}
On the other hand, by the Markov inequality, we have
\[ \bP(Y_n>0)\leq \E(Y_n).\]
Note that $\E(Y_n)=m^n$. Thus, letting $n\to \infty$, we get
\[\limsup_{n\to\infty}\frac{1}{n}\log\bP(Y_n>0)\leq \log m,\]
which, combined with (\ref{lowerbound}), finishes the proof.
\end{proof}

With Lemma \ref{lemGW}, we can prove Proposition \ref{log}. Here we mainly use the strategy showed in \cite[Theorem 2]{biggins1} to construct a truncated branching random walk that controls the lower bound of the probability $\bP(M_n\geq nx)$. For $u,v\in \mathbb{T}$, we denote by $uv:=(u^{(1)},\ldots,u^{(|u|)},v^{(1)},\ldots,v^{(|v|)})$ the concatenation of $u$ and $v$.
\begin{proof}[Proof of Proposition \ref{log}]
For the upper bound, observe that for $x\in\R$, we have $\{M_n\geq nx\}=\{Z_n([nx,\infty))\geq 1\}.$ By Markov inequality,
\[\bP(Z_n([nx,\infty))\geq 1)\leq \E(Z_n([nx,\infty)))=\E\left(\sum_{|u|=n}1_{\{V(u)\geq nx\}}\right).\]
Thus, for any $\theta\geq 0$, by many-to-one formula, we have
\[\E(Z_n([nx,\infty)))\leq\E\left(\sum_{|u|=n}e^{\theta V(u)-nx\theta}\right)= e^{-n(\theta x-\psi(\theta))}.\]
Optimizing the above equation with respect to $\theta$, it implies
\[\bP(M_n\geq nx)\leq e^{-n\psi^*(x)}.\]

 Next, we bound $\bP(M_n\geq nx)$ from below. If $\psi^*(x)=\infty$, then $\bP(M_n\geq nx)=0$ completes the proof. We therefore assume that $\psi^*(x)<\infty$. We know that $\psi^*(x)>0$ because $\psi^*$ is convex on $\R$ and $x>x^*$. Similarly to \cite[Theorem 2]{biggins1}, fix $k\in\N$, we construct a branching random walk $Z^{(k)}$ as follows. The first generation of $Z^{(k)}$ will consist only of those particles in the $k$th generation of the branching random walk $Z$ with positions to the right of $kx$. The second generation is formed by applying the same procedure to the branching random walks initiated by each of these particles and so on. More precisely, for $k\in\N$ and $i\geq 1$, we denote by
\[A_0^k:=\{\varnothing\}\text{ and }A_i^k:=\{uy\in \mathbb{T}:u\in A_{i-1}^k,|y|=k,V(uy)-V(u)\geq kx\}\]
the collections of particles in each generation of the new branching random walk $Z^{(k)}$, which yields that $Z_i^{(k)}(\R)=\#A_i^k$. It is obvious to see that $(Z_n^{(k)}(\R))_{n\geq 0}$ is a Galton-Watson process with $Z_0^{(k)}(\R)=1$ and
$\E(Z^{(k)}_1(\R))=\E(Z_k([kx,\infty))).$ By the independence of the branching events, we know that
\[
\begin{split}
\bP(M_n\geq nx)&\geq \bP(M_{n-k\lfloor\frac{n}{k}\rfloor}\geq (n-k\lfloor\frac{n}{k}\rfloor)x)\bP(Z^{(k)}_{\lfloor\frac{n}{k}\rfloor}(\R)>0)\\
&\geq \inf_{1\leq i\leq k }\bP(M_i\geq ix)\bP(Z^{(k)}_{\lfloor\frac{n}{k}\rfloor}(\R)>0),
\end{split}
\]
where $\lfloor x\rfloor$ means the integer part of $x$.
As $\psi(0)>0$, according to \cite[Theorem 1]{biggins1}, we have
\begin{equation}\label{cramer}
\lim_{n\to\infty}\frac{1}{n}\log \E(Z_n([nx,\infty)))=-\psi^*(x)\in(-\infty,0),
\end{equation}
which implies that for $k$ large enough, $\E(Z^{(k)}_1(\R))\in(0,1).$

By Lemma \ref{lemGW}, we get
\[\liminf_{n\to\infty}\frac{1}{n}\log\bP(M_n\geq nx)\geq \lim_{n\to\infty}\frac{\lfloor\frac{n}{k}\rfloor}{n}\frac{1}{\lfloor\frac{n}{k}\rfloor}\log\bP(Z^{(k)}_{\lfloor\frac{n}{k}\rfloor}(\R)>0)=\frac{1}{k}\log\E(Z_k([kx,\infty))),\]
where the first inequality follows from $\inf_{1\leq i\leq k}\bP(M_i\geq ix)>0$. By (\ref{cramer}), letting $k\to\infty$, we get
\[\liminf_{n\to\infty}\frac{1}{n}\log\bP(M_n\geq nx)\geq-\psi^*(x),\]
which completes the proof of lower bound in the case that $\psi(0)>0$.

On the other hand, if $\psi(0)\leq 0$, we can denote by $\mu$ the intensity measure of $Z_1$ defined by $\mu(A):=\E(Z_1(A))$ for any $A\in\mathscr{B}(\R)$. By many-to-one formula, we have
 \[\mu^{*n}(A)=\E(Z_n(A)),\]
where $\mu^{*n}$ means the $n$-fold convolution of $\mu$.
 As $\mu(\R)=\E(Z_1(\R))=e^{\psi(0)}<\infty$, $\mu/\mu(\R)$ is a probability measure on $\R$. By Cram\'er's theorem, we know that for any $x>\frac{1}{\mu(\R)}\int_{\R}s\mu(\mathrm{d}s)=\frac{\E(\sum_{|u|=1}V(u))}{\E(Z_1(\R))}$, we have
\[\lim_{n\to\infty}\frac{1}{n}\log\E(Z_n([nx,\infty)))=\lim_{n\to\infty}\frac{1}{n}\log \frac{\mu^{*n}([nx,\infty))}{\mu(\R)^n}+\log\mu(\R)=-\psi^*(x)\in(-\infty,0).\]
With the same arguments as in the case that $\psi(0)>0$, we have
\[\liminf_{n\to\infty}\frac{1}{n}\log \bP(M_n\geq nx)\geq -\psi^*(x),\quad\text{for any } x>\frac{\E(\sum_{|u|=1}V(u))}{\E(Z_1(\R))},\]
which completes the proof of lower bound.
\end{proof}

\section{Some extra properties of the extremal process of conditioned branching random walk}\label{finitenumber}
In Section \ref{critical}, we will discuss when $D(\theta)$ is finite. We give an alternative description of $D(\theta)$ in Section \ref{alternative}. Finally, we prove the weak continuity of $D(\theta)$ in Section \ref{limitprocess}.
\subsection{Subcritical branching random walk}\label{critical}
 Recall that we do not assume the branching random walk $Z$ to be supercritical. Under some extra conditions, we prove that $D(\theta)$ is finite (infinite) if $Z$ is subcritical (critical) and satisfies a $L\log L$ condition (second moment condition).
\begin{proposition}\label{subcritical}
Assume there exists $\theta>0$ such that $\psi(\theta)<\infty$ and $(\ref{as3})$. If $Z$ is subcritical and satisfies $\E(Z_1(\R)\log_+Z_1(\R))<\infty$, then $\bP(D_\infty^\theta(\R)<\infty)=1.$
\end{proposition}
\begin{proof}
Because $Z$ is subcritical, then we know that $\psi(0)<0$, which implies $\E(Z_1(\R))\in(0,1)$. According to the definition of $Z_n$ in (\ref{Zn}), we know that $(Z_n(\R))_{n\geq 0}$ is a subcritical GW process with $Z_0(\R)=1$.  By Lemma \ref{lemGW}, as $n\to\infty$,
\[\frac{1}{n}\log\bP(Z_n(\R)>0)\to \log\E(Z_1(\R))\in(-\infty,0).\]
Recall the definitions of $(b_k(\ell):\ell\geq 1)_{k\geq 1}$ and $(V^{(i,k)}(u):u\in\mathbb{T})_{i\geq 1,k\geq 1}$. Let $Z^{(i,k)}$ be the number of particles of the branching random walk $(V^{(i,k)}(u):u\in\mathbb{T})$ at time $k-1$, i.e.
\[Z^{(i,k)}:=\sum_{|u|=k-1}1_{\{V^{(i,k)}(u)\in\R\}}.\]
Define
\[B_k:=\sum_{\ell\geq 1}\delta_{b_k(\ell)}\]
the counting measure formed by $(b_k(\ell):\ell\geq 1)$.
Then
\[\bP(D_\infty^\theta(\R)=\infty)\leq\bP\left(\sum_{k\geq 2}\sum_{i= 1}^{B_k(\R)}Z^{(i,k)}=\infty\right).\]
Observe that there exists $\vep_1\in(0,1)$, for $k$ large enough, we have
\[
\begin{split}
\bP\left(\sum_{i=1}^{B_k(\R)}Z^{(i,k)}>0\right)&=\bP\left(\sum_{i=1}^{B_k(\R)}Z^{(i,k)}>0,B_k(\R)>e^{ck}\right)+\bP\left(\sum_{i=1}^{B_k(\R)}Z^{(i,k)}>0,B_k(\R)\leq e^{ck} \right)\\
&\leq \bP(\log_+B_k(\R)>ck)+e^{ck}\bP(Z_{k-1}>0)\\
&\leq \bP(\log_+B_k(\R)>ck)+Ce^{ck}e^{k(1-\vep_1)\log \E(Z_1(\R))},
\end{split}
\]
where $c$ is a positive constant in $(0,-(1-\vep_1)\log\E(Z_1(\R)))$. Because $x\log(1+y)\leq x\log(1+x)+y\log(1+y)$ for any $x,y\geq0$, we get
\[\E\left(\sum_{|u|=1}e^{\theta V(u)}\log_+Z_1(\R)\right)\leq \E(Z_1(\R)\log_+(Z_1(\R)))+\E\left(\sum_{|u|=1}e^{\theta V(u)}\log_+(\sum_{|u|=1}e^{\theta V(u)})\right)<\infty,\]
which implies that $\E(\log_+(B_1(\R)))<\infty$. We therefore know that
\[\sum_{k=2}^\infty\bP\left(\sum_{i=1}^{B_k(\R)}Z^{(i,k)}>0\right)<\infty.\]
Thus, by the Borel-Cantelli Lemma, we conclude that almost surely for $k$ large enough,
\[\sum_{i=1}^{B_k(\R)}Z^{(i,k)}=0.\]
Because $\E Z_1(\R)<\infty$, we have $Z_1(\R)<\infty$ almost surely,
which implies that for any $i,k\geq 1$,
\[B_k(\R)<\infty\quad\text{and}\quad Z^{(i,k)}<\infty,\quad\text{a.s.}\]
Thus, we know that for any $n>0$, almost surely
\[\sum_{k=2}^n\sum_{i=1}^{B_k(\R)}Z^{(i,k)}<\infty.\]
Thus, we have
\[\bP\left(\sum_{k=2}^\infty\sum_{i=1}^{B_k(\R)}Z^{(i,k)}=\infty\right)=0,\]
which proves that
\[\bP(D_\infty^\theta(\R)<\infty)=1.\qedhere\]
\end{proof}
\begin{proposition}
Assume that there exists $\theta>0$ satisfying $\psi(\theta)<\infty$. If $Z$ is critical and satisfies $\E(Z_1(\R)^2)<\infty$, then $\bP(D_\infty^\theta(\R)=\infty)=1$.
\end{proposition}
\begin{proof}
For $k\geq 1$, we define
\[C_k:=\{B_k(\R)\geq 2\text{ and there exists }i\neq w^{(k)}\text{ such that } Z^{(i,k)}>0\},\]
 where $Z^{(i,k)}$ is defined in Proposition \ref{subcritical}. Let
\[s_k:=(w^{(k)}-1)1_{\{w^{(k)}>1\}}+2\times1_{\{w^{(k)}=1\}}.\]
Observe that
\[\bP(C_k)\geq \bP(B_k(\R)\geq 2,Z^{(s_k,k)}>0).\]
Since that $(B_k(\R),w^{(k)})$ and $(Z^{(i,k)})_{i\geq 1}$ are independent and $(Z^{(i,k)})_{i\geq 1}$ are identical distributed, we have
\[\bP(B_k(\R)\geq 2,Z^{(s_k,k)}>0)=\bP(B_k(\R)\geq 2)\bP(Z^{(s_k,k)}>0)=\bP(B_1(\R)\geq 2)\bP(Z_{k-1}(\R)>0),\]
where the last equation follows from that $(B_k(\R))_{k\geq 1}$ are identical distributed and $Z^{(s_k,k)}$ is distributed as $Z_{k-1}(\R)$.
Because $Z$ is critical and $\E(Z_1(\R)^2)<\infty$, by \cite[P20]{athreya}, we have
\[n\bP(Z_n(\R)>0)\to\frac{2}{\E(Z_1(\R)^2)},\quad \text{as }n\to\infty.\]
Thus, there exists some positive constant $C$ such that
\[\sum_{k=1}^\infty\bP(C_k)\geq C\bP(B_1(\R)\geq 2)\sum_{k=1}^\infty \frac{1}{k},\]
which, by that $\bP(B_1(\R)\geq2)>0$, yields
\[\sum_{k=1}^\infty\bP(C_k)=\infty.\]
Because $(B_k(\R),w^{(k)},Z^{(i,k)},i\geq 1)_{k\geq 1}$ are independent, we obtain that $(C_k)_{k\geq 1}$ are independent. By Borel-Cantelli Lemma, we have almost surely, infinitely many $k$ such that $B_k(\R)\geq 2$ and there exists $i\neq w^{(k)}$ satisfying $Z^{(i,k)}>0$. Observe that
\[D_\infty^\theta(\R)\geq \sum_{k=1}^\infty\sum_{1\leq i\leq B_k(\R),i\neq w^{(k)}}Z^{(i,k)},\]
we therefore have
\[\bP(D_\infty^\theta(\R)=\infty)=1.\qedhere\]
\end{proof}

\subsection{Extremal process as a conditioned point measure}\label{alternative}
In this section, we give an alternative description of $D(\theta)$. We define a total order $<$ in $\mathbb{T}$ called the lexicographical order. For $u,v\in\mathbb{T}$, define $r(u,v):=\sup\{k\geq 0:u_k=v_k\}$ the generation of the most recent common ancestor of $u$ and $v$. Then we say that $u<v$
 if one of conditions below is satisfied:
\begin{enumerate}[1)]
\item $r(u,v)=|u|$ and $|u|<|v|$.
\item $r(u,v)<|u|$ and $u^{(r(u,v)+1)}<v^{(r(u,v)+1)}$.
\end{enumerate}
When we see the branching random walk from the minimal particle located at the maximal position at time $n$ for the lexicographical order, we can get the next lemma, which is similar to Lemma \ref{lemD}.
\begin{lemma}
Let $\theta>0$ such that $\psi(\theta)<\infty$ and $\psi'(\theta)<\infty$. For any non-negative measurable function F, f and $n\geq 1$, we have
\[
\begin{split}
&\E(F(\mathcal{E}_n)f(M_n-n\psi'(\theta))1_{\{Z_n(\R)>0\}})\\
=&e^{n(\psi(\theta)-\theta\psi'(\theta))}\E(e^{-\theta (S_n-n\psi'(\theta))}F(D_n^\theta)f(S_n-n\psi'(\theta))1_{\{D_n^\theta((0,\infty))=0,\bar{D}_n^\theta=0\}}),
\end{split}
\]
where
\[\bar{D}_n^\theta:=\sum_{k=1}^n\sum_{i< w^{(k)}}\sum_{|u|=k-1}1_{\{b_k(i)+V^{(i,k)}(u)=S_k\}}.\]
\end{lemma}
\begin{proof}
Define $m_n$ to be the minimal element of $N_n$ for the lexicographical order. By spine decomposition theorem, with similar arguments as in the proof of Lemma \ref{lemD}, we have
\[
\begin{split}
&\E[F(\mathcal{E}_n)f(M_n-\psi'(\theta)n)1_{\{Z_n(\R)>0\}}]\\
=&\E\left(\sum_{|u|=n}F(\mathcal{E}_n)1_{\{u=m_n\}}f(M_n-\psi'(\theta)n)1_{\{Z_n(\R)>0\}}\right)\\
=&\tilde{\E}(e^{-\theta V(\xi_n)+n\psi(\theta)}F(\mathcal{E}_n^{\xi_n})1_{\{\mathcal{E}_n^{\xi_n}((0,\infty))=0\}}1_{\{\sum_{|u|=n,u<\xi_n}1_{\{V(u)=V(\xi_n)\}}=0\}}f(V(\xi_n)-n\psi'(\theta))).
\end{split}
\]
 Observe that
\[
\begin{split}
&\sum_{|u|=n,u<\xi_n}1_{\{V(x)=V(\xi_n)\}}\\
=&\sum_{k=1}^{n}\sum_{i=1}^{\xi^{(k)}-1}\sum_{|u|=n-k}1_{\{V(\xi_{k-1}i)-V(\xi_{k-1})+V(\xi_{k-1}iu)-V(\xi_{k-1}i)=V(\xi_n)-V(\xi_{k-1})\}}.
\end{split}
\]
With the same arguments as in the proof of Lemma \ref{lemD}, we have that $(D_n^\theta,\bar{D}_n^\theta,S_n)$ under $\bP$ is distributed as
\[(\mathcal{E}_n^{\xi_n},\sum_{|u|=n,u<\xi_n}1_{\{V(u)=V(\xi_n)\}},V(\xi_n))\]
under $\tilde{\bP}$.
Therefore
\[
\begin{split}
&\tilde{\E}(e^{-\theta V(\xi_n)+n\psi(\theta)}F(\mathcal{E}_n^{\xi_n})1_{\{\mathcal{E}_n^{\xi_n}((0,\infty))=0\}}1_{\{\sum_{|u|=n,u<\xi_n}1_{\{V(u)= V(\xi_n)\}}=0\}}f(V(\xi_n)-n\psi'(\theta)))\\
=&\E(e^{-\theta S_n+n\psi(\theta)}F(D_n^\theta)f(S_n-n\psi'(\theta))1_{\{D_n^\theta((0,\infty))=0\}}1_{\{\bar{D}_n^\theta=0 \}})\\
=&e^{n(\psi(\theta)-\theta\psi'(\theta))}\E(e^{-\theta S_n+n\theta\psi'(\theta)}F(D_n^\theta)f(S_n-n\psi'(\theta))1_{\{D_n^\theta((0,\infty))=0\}}1_{\{\bar{D}_n^\theta=0\}}),
\end{split}
\]
which completes the proof.
\end{proof}
\begin{remark}\label{another case}
With the same arguments as in the proof of Lemma \ref{part1}, we can also prove
\[\E\left(\frac{1}{D_\infty^\theta(\{0\})}e^{-D_\infty^\theta(\phi)}1_{\{D_\infty^\theta((0,\infty))=0\}}\right)=\E(e^{-D_\infty^\theta(\phi)}1_{\{D_\infty^\theta((0,\infty))=0\}}1_{\{\bar{D}_\infty^\theta=0\}}).\]
By the proof of Theorems \ref{th2} and \ref{th1}, we get
\[C(\theta)=\bP(D_\infty^\theta((0,\infty))=0,\bar{D}_\infty^\theta=0)\]
and $D(\theta)$ is distributed as $(D_\infty^\theta|D_\infty^\theta((0,\infty))=0,\bar{D}_\infty^\theta=0).$
\end{remark}
\subsection{Continuity of the extremal process}\label{limitprocess}
In this section, we prove that $D(\theta)$ in Theorem \ref{th1} is continuous in distribution with respect to $\theta$ under some stronger conditions.
First, we introduce an important lemma.
\begin{lemma}\label{continuity}
Fix $\theta_0>0$. If there exists $\delta\in(0,\theta_0)$ satisfying for any $\theta\in[\theta_0-\delta,\theta_0+\delta]$, $\theta$ satisfies assumptions (\ref{as1}), (\ref{as2}) and
\[\E\left[\sum_{|u|=1}e^{\theta V(u)}\left(\log_+\left(\sum_{|u|=1}e^{\theta V(u)}\right)\right)^{1+\delta}\right]<\infty.\]
There exists $\vep>0$ s.t. as $n\to\infty$, we have
\[\sup_{\theta\in[\theta_0-\delta,\theta_0+\delta]}\bP(A_n(\theta,\vep)^c)\to0,\]
where
\[A_n(\theta,\vep):=\left\{\max_{\ell\geq n}\frac{\max_{k\geq 1}\{b_\ell(k)+\max_{|u|=\ell-1}\{V^{(k,\ell)}(u)\}\}-S_{\ell}}{\ell}<-\vep\right\}\]
and we should emphasize that $b_\ell(k)$ and $S_\ell$ depend on $\theta$ even if we do not give them a subscript of $\theta$.
\end{lemma}
\begin{proof}
With the same arguments as in Lemma \ref{well}, we have
\[
\begin{split}
&\bP(A_n(\theta,\vep)^c)\\
\leq &\bP\left(\max_{\ell\geq n}\frac{\max_{k\geq 1}\{b_\ell(k)+\max_{|u|=\ell-1}\{V^{(k,\ell)}(u)\}\}-S_{\ell}}{\ell}\geq-\vep\right)\\
 \leq &e^{\psi(\theta)}\sum_{\ell\geq n}e^{-\ell(\theta\psi'(\theta)-\psi(\theta)-c-2\vep\theta)}+\sum_{\ell\geq n}\bP\left(\log_+\left(\sum_{k\geq 1}e^{\theta b_1(k)}\right)>c\ell\right)+\sum_{\ell\geq n}\bP(S_\ell<(\psi'(\theta)-\vep)\ell),
 \end{split}
 \]
 where $c$ is a positive constant whose value is fixed later.
 Observe that $\theta\psi'(\theta)-\psi(\theta)$ is increasing on $[\theta_0-\delta,\theta_0+\delta]$ by convexity and is positive, then
 \[c_1:=\inf_{\theta\in[\theta_0-\delta,\theta_0+\delta]}\{\theta\psi'(\theta)-\psi(\theta)\}=(\theta_0-\delta)\psi'(\theta_0-\delta)-\psi(\theta_0-\delta)\in(0,\infty).\]
 With the same arguments as above, we can get
 \[c_2:=\inf_{\theta\in[\theta_0-\delta,\theta_0+\delta]}\{\psi'(\theta)-\frac{\psi(\theta)}{\theta}\}\in[\frac{c_1}{\theta_0+\delta},\frac{(\theta_0+\delta)\psi'(\theta_0+\delta)-\psi(\theta_0+\delta)}{\theta_0-\delta}]\subset(0,\infty).\]
 Meanwhile, using the continuity of $\psi''$ on $[\theta_0-\delta,\theta_0+\delta]$ and that $\psi''(\theta)>0$ for any $\theta\in[\theta_0-\delta,\theta_0+\delta]$, we know that
 \[c_3:=\sup_{\theta\in[\theta_0-\delta,\theta_0+\delta]}\psi''(\theta)\in(0,\infty).\]
 For $\vep_1\in(0,\frac{\delta}{2})$, by mean-valued theorem, for any $\theta\in[\theta_0-\frac{\delta}{2},\theta_0+\frac{\delta}{2}]$, we have
 \[\psi'(\theta)-\frac{\psi(\theta)-\psi(\theta-\vep_1)}{\vep_1}=\psi'(\theta)-\psi'(\theta(\vep_1))=(\theta-\theta(\vep_1))\psi''(\theta_{\vep_1})\leq \vep_1c_3,\]
 where $\theta(\vep_1)$ is a constant in $(\theta-\vep_1,\theta)$ and $\theta_{\vep_1}$ is a constant in $(\theta(\vep_1),\theta)$.
 By many-to-one formula and the spine decomposition theorem, we have
 \[
 \begin{split}
 \bP(S_\ell<(\psi'(\theta)-\vep)\ell)\leq \E(e^{\vep_1(\psi'(\theta)-\vep)\ell-\vep_1 S_\ell})&=e^{\vep_1(\psi'(\theta)-\vep)\ell}\E\left(\sum_{|u|=\ell}e^{(\theta-\vep_1)V(u)-\ell\psi(\theta)}\right)\\
 &=e^{-\ell(\psi(\theta)-\psi(\theta-\vep_1)-\vep_1(\psi'(\theta)-\vep))}.
 \end{split}
 \]
Set
 \[c=\frac{1}{2}c_1,\quad \vep_1=\min\{\frac{1}{8}\frac{c_2}{c_3},\frac{\delta}{2}\}\quad\text{and}\quad\vep=\frac{1}{6}c_2,\]
 for any $\theta\in[\theta_0-\delta,\theta_0+\delta]$, we have
 \[\theta\psi'(\theta)-\psi(\theta)-c-2\theta\vep\geq\frac{1}{6}c_2(\theta_0-\delta)\quad\text{and}\quad \psi(\theta)-\psi(\theta-\vep_1)-\vep_1(\psi'(\theta)-\vep)\geq\frac{1}{24}\vep_1c_2. \]
Observe that
\[\bP\left(\log_+\left(\sum_{k\geq 1}e^{\theta b_1(k)}\right)>c\ell\right)\leq \frac{1}{c^{1+\delta}\ell^{1+\delta}}\E\left(\sum_{|u|=1}e^{\theta V(u)}\left(\log_+\left(\sum_{|u|=1}e^{\theta V(u)}\right)\right)^{1+\delta}\right).\]
Let $g(\theta):=\E(\sum_{|u|=1}e^{\theta V(u)}(\log_+(\sum_{|u|=1}e^{\theta V(u)}))^{1+\delta}),\theta\in[\theta_0-\delta,\theta_0+\delta]$. By dominated convergence theorem, it is straightforward to see that $g$ is continuous on $[\theta_0-\delta,\theta_0+\delta]$, which yields
\[c_4:=\sup_{\theta\in[\theta_0-\delta,\theta_0+\delta]}g(\theta)\in(0,\infty).\]
Besides, by the continuity of $\psi$ on $[\theta_0-\delta,\theta_0+\delta]$, we know that
\[c_5:=\sup_{\theta\in[\theta_0-\delta,\theta_0+\delta]}\psi(\theta)\in(0,\infty).\]
In final, for $n$ large enough, for any $\theta\in[\theta_0-\delta,\theta_0+\delta]$, we have
\[
\bP(A_n(\theta,\vep)^c)\leq e^{c_5}\sum_{\ell\geq n}e^{-\frac{1}{6}c_2(\theta_0-\delta)\ell} +\sum_{\ell\geq n} \frac{c_4}{c^{1+\delta}\ell^{1+\delta}}+\sum_{\ell\geq n}e^{-\frac{1}{24}\vep_1c_2\ell},
\]
which concludes that
\[\lim_{n\to\infty}\sup_{\theta\in[\theta_0-\delta,\theta_0+\delta]}\bP(A_n(\theta,\vep)^c)\to0.\qedhere\]
\end{proof}
Now we use the Lemma \ref{continuity} to give the proof of continuity.
\begin{proposition}
Consider a non-lattice branching random walk and fix $\theta_0>0$. If there exists $\delta\in(0,\theta_0)$ satisfying for any $\theta\in[\theta_0-\delta,\theta_0+\delta]$, $\theta$ satisfies assumptions (\ref{as1}), (\ref{as2}) and
\[\E\left[\sum_{|u|=1}e^{\theta V(u)}\left(\log_+\left(\sum_{|u|=1}e^{\theta V(u)}\right)\right)^{1+\delta}\right]<\infty,\]
then $(D_\infty^\theta|D_\infty^\theta((0,\infty))=0,\bar{D}_\infty^\theta=0)$ is continuous in distribution at $\theta_0$.
\end{proposition}
\begin{proof}
It is sufficient to prove that for any non-negative function $\phi\in C_c(\R)$,
\[\lim_{\theta\to\theta_0}\E(e^{-D_\infty^\theta(\phi)}1_{\{D_\infty^\theta((0,\infty))=0,\bar{D}_\infty^\theta=0\}})=\E(e^{-D_\infty^{\theta_0}(\phi)}1_{\{D_\infty^{\theta_0}((0,\infty))=0,\bar{D}_\infty^{\theta_0}=0\}}).\]
For $\vep>0$, $\theta\in[\theta_0-\delta,\theta_0+\delta]$, let $A_n(\theta,\vep)$ be defined as in Lemma \ref{continuity}.
Observe that there exists $N_1\in\N$ s.t. $-\vep N_1<\inf\text{supp}(\phi)$, then for $\theta\in[\theta_0-\delta,\theta_0+\delta]$ and $n\geq N_1$, by Remark \ref{another case}, we have
 \[
 \begin{split}
 &\E(e^{-D_\infty^\theta(\phi)}1_{\{D_\infty^\theta((0,\infty))=0,\bar{D}_\infty^\theta=0\}})\\
 =&\E\left(\frac{1}{D_\infty^\theta(\{0\})}e^{-D_\infty^\theta(\phi)}1_{\{D_\infty^\theta((0,\infty))=0\}}\right)\\
 =&\E\left[\frac{1}{D_\infty^\theta(\{0\})}e^{-D_\infty^\theta(\phi)}1_{\{D_\infty^\theta((0,\infty))=0\}}1_{A_n(\theta,\vep)}\right]+\E\left[\frac{1}{D_\infty^\theta(\{0\})}e^{-D_\infty^\theta(\phi)}1_{\{D_\infty^\theta((0,\infty))=0\}}1_{A_n(\theta,\vep)^c}\right]\\
 =&\E\left(\frac{1}{D_n^\theta(\{0\})}e^{-D_n^\theta(\phi)}1_{\{D_n^\theta((0,\infty))=0\}}\right)-\E\left(\frac{1}{D_n^\theta(\{0\})}e^{-D_n^\theta(\phi)}1_{\{D_n^\theta((0,\infty))=0\}}1_{A_n(\theta,\vep)^c}\right)\\
 &+\E\left(\frac{1}{D_\infty^\theta(\{0\})}e^{-D_\infty^\theta(\phi)}1_{\{D_\infty^\theta((0,\infty))=0\}}1_{A_n(\theta,\vep)^c}\right),
 \end{split}
 \]
 which yields that
 \[
 \begin{split}
 &|\E(e^{-D_\infty^\theta(\phi)}1_{\{D_\infty^\theta((0,\infty))=0,\bar{D}_\infty^\theta=0\}})-\E(e^{-D_\infty^{\theta_0}(\phi)}1_{\{D_\infty^{\theta_0}((0,\infty))=0,\bar{D}_\infty^{\theta_0}=0\}})|\\
 \leq &2\bP(A_n(\theta,\vep)^c)+\left|\E\left(\frac{1}{D_n^\theta(\{0\})}e^{-D_n^\theta(\phi)}1_{\{D_n^\theta((0,\infty))=0\}}\right)-\E(e^{-D_\infty^{\theta_0}(\phi)}1_{\{D_\infty^{\theta_0}((0,\infty))=0,\bar{D}_\infty^{\theta_0}=0\}})\right|.
 \end{split}
 \]
 By Lemma \ref{lemD}, we have
 \[\E\left(\frac{1}{D_n^\theta(\{0\})}e^{-D_n^\theta(\phi)}1_{\{D_n^\theta((0,\infty))=0\}}\right)=\E(e^{\theta M_n-n\psi(\theta)}e^{-\mathcal{E}_n(\phi)}1_{\{Z_n(\R)>0\}}).\]
 Observe that $\psi$ is continuous on $[\theta_0-\delta,\theta_0+\delta]$, by dominated convergence theorem, we have
 \[
 \begin{split}
 \lim_{\theta\to\theta_0}\E\left(e^{\theta M_n-n\psi(\theta)}e^{-\mathcal{E}_n(\phi)}1_{\{Z_n(\R)>0\}}\right)&=\E\left(e^{\theta_0M_n-n\psi(\theta_0)}e^{-\mathcal{E}_n(\phi)}1_{\{Z_n(\R)>0\}}\right)\\
 &=\E\left(\frac{1}{D_n^{\theta_0}(\{0\})}e^{-D_n^{\theta_0}(\phi)}1_{\{D_n^{\theta_0}((0,\infty))=0\}}\right),
 \end{split}
 \]
 which implies that
 \[
 \begin{split}
 \lim_{n\to\infty}\lim_{\theta\to\theta_0}\E\left(\frac{1}{D_n^\theta(\{0\})}e^{-D_n^\theta(\phi)}1_{\{D_n^\theta((0,\infty))=0\}}\right)&=\E\left(\frac{1}{D_\infty^{\theta_0}(\{0\})}e^{-D_\infty^{\theta_0}(\phi)}1_{\{D_\infty^{\theta_0}((0,\infty))=0\}}\right)\\
 &=\E(e^{-D_\infty^{\theta_0}(\phi)}1_{\{D_\infty^{\theta_0}((0,\infty))=0,\bar{D}_\infty^{\theta_0}=0\}}).
 \end{split}
 \]
 On the other hand, by Lemma \ref{continuity}, we have
 \[\lim_{n\to\infty}\sup_{\theta\in[\theta_0-\delta,\theta_0+\delta]}\bP(A_n(\theta,\vep)^c)=0.\]
Thus, we have
\[\lim_{\theta\to\theta_0}\E(e^{-D_\infty^\theta(\phi)}1_{\{D_\infty^\theta((0,\infty))=0,\bar{D}_\infty^\theta=0\}})=\E(e^{-D_\infty^{\theta_0}(\phi)}1_{\{D_\infty^{\theta_0}((0,\infty))=0,\bar{D}_\infty^{\theta_0}=0\}}),\]
which completes the proof.
\end{proof}
\noindent\textbf{Acknowledgements:}
I wish to thank my supervisor Bastien Mallein for introducing me to this subject and constantly finding the time for useful discussions and advice. I am supported by China Scholarship Council (No.202306040031).

\end{document}